\documentclass[12pt,a4paper]{amsart}

\usepackage{amsmath}
\usepackage{amssymb}
\usepackage{amsthm}
\usepackage{latexsym}
\usepackage[enableskew]{youngtab}
\usepackage{pgf,tikz}
\usetikzlibrary{arrows}

\usepackage{amscd}
\usepackage{graphicx}
\definecolor{zzttqq}{rgb}{0.6,0.2,0}
\definecolor{qqqqff}{rgb}{0,0,1}

\theoremstyle{plain}
\newtheorem{theorem}{Theorem}[section]
\newtheorem{corollary}[theorem]{Corollary}
\newtheorem{lemma}[theorem]{Lemma}
\newtheorem{proposition}[theorem]{Proposition}

\theoremstyle{definition}
\newtheorem{definition}{Definition}[section]
\newtheorem{algorithm}{Algorithm}
\newtheorem{procedure}{Procedure}

\theoremstyle{remark}
\newtheorem{remark}{Remark}[section]
\newtheorem{fig}{Figure}[section]
\newtheorem{example}{Example}[section]

\numberwithin{equation}{section}

\numberwithin{table}{section}

\numberwithin{figure}{section}

\def\ww{{\bf w}}
\def\nn{{\bf n}}
\def\sup{{\rm supp}}

\title[]
{Multiplicity--free Skew
Schur functions with full interval support}

\author{Olga Azenhas}
\address{CMUC, Department of Mathematics, University of Coimbra, Apartado 3008,
3001--454 Coimbra, Portugal} \email{oazenhas@mat.uc.pt}

\author{Alessandro Conflitti}
\email{alessandro.conflitti@gmail.com}

\author{Ricardo Mamede}
\address{CMUC, Department of Mathematics, University of Coimbra, Apartado 3008,
3001--454 Coimbra, Portugal} \email{mamede@mat.uc.pt}



\thanks{ This work was partially supported by the Centre for Mathematics of the
University of Coimbra -- UID/MAT/00324/2013, funded by the Portuguese
Government through FCT/MCTES and co-funded by the European Regional
Development Fund through the Partnership Agreement PT2020.
 The first author (OA) was also partially supported by  the FCT sabbatical
grant SFRH/BSAB/113584/2015,
 and wishes to acknowledge the hospitality of the University of Vienna where her sabbatical leaving took place. This work has been conducted when the second author (AC) was a member of the Centre for Mathematics of the University of Coimbra.}

\begin{document}

\begin{abstract}
It is known that the Schur expansion of a skew Schur function runs over the
interval of partitions, equipped with dominance order, defined by the
least and the most dominant Littlewood--Richardson filling of the skew shape.
We characterise skew Schur functions (and therefore the product of two Schur
functions) which are multiplicity--free and  the resulting Schur expansion runs
over the whole interval of partitions, i.e. skew Schur functions having
Littlewood--Richardson coefficients always equal to $1$ over the
full interval.
\end{abstract}

\maketitle

\section{Introduction and statement of results}

The ring of symmetric functions has a linear basis of Schur functions $s_\nu$ indexed by partitions $\nu$. Skew Schur functions $s_{\lambda/\mu}$ are symmetric functions indexed by skew partitions
 $\lambda/\mu$ and they can be expressed as a linear combination of Schur functions by means of the Littlewood-Richardson coefficients $c_{\mu,\nu}^\lambda$,  which are non negative integers,
$$s_{\lambda/\mu}=\sum_{\nu} c_{\mu,\nu}^\lambda s_\nu.$$
In particular, the product of two Schur functions is governed by these coefficients,
$s_\mu s_{\nu}=\sum_{\lambda} c_{\mu,\nu}^\lambda s_\lambda$.  In representation theory this basis is important because its elements  occur as characters of the general linear group $GL_n$, and they correspond to characters of the symmetric group via the Frobenius map. Schur
functions also have an intersection-theoretic interpretation as representatives of Schubert classes  in  the cohomology ring of a Grassmannian.
 Thus Littlewood-Richardson
coefficients  amount to multiplicities of  irreducible representations, as well as   to multiplicities in the decomposition of the cup product of Schubert classes.

 For any skew shape $A$,  the support of $A$ (or $s_A$) is defined to be the set of the conjugate of those
 partitions $\nu$ such that the Schur function $s_\nu$ appears with a  positive
 integer coefficient in the Schur expansion of $s_A$. It is well known that the support of $A$, considered as
a subposet of the dominance lattice, has a top element, $\nn$,  the
conjugate of the partition formed by the row lengths of $A$,
 and a bottom element, $\ww$,  the partition formed by the column lengths of $A$.
 More precisely, the Schur  expansion above  with $A=\lambda/\mu$ can be written within the interval $ [\ww, \nn]$ in the dominance lattice, as
\begin{equation}\label{motiv} s_{\lambda/\mu}=\sum_{\nu'\in[\ww,\nn]} c_{\mu,\nu}^\lambda s_\nu.\end{equation}
A  very general problem is the classification of the shapes $A$ whose support consists
of the whole interval $ [\ww, \nn]$ in the dominance lattice.
In other words, given the triple of partitions $(\mu,\nu,\lambda)$, with $\mu\subseteq \lambda$, we address the question on under which conditions  we have,
$c_{\mu\,\nu}^{\lambda}>0$ {if and only if } $ \nu'\in[\ww,\nn]$, or
$c_{\mu\,\nu}^{\lambda}>0$ if and only if   $ \mu\cup\nu\preceq \lambda\preceq \mu+\nu.$ Efforts to this classification in the case of ribbon shapes have been  in progress \cite{Mn, MnvW1,GaHaSri}. This problem is equivalent to the classification of skew
characters of the symmetric group and  Schubert products
which obey the same properties \cite{gut,KWvS}.
Here we answer the  question for which skew shapes $A$
 it is the case that the Schur function $s_A$ can be expressed as
\begin{equation}\label{free} s_{A}=\sum_{\nu'\in[\ww,\nn]}  s_\nu,
 \end{equation}
 that is, with the coefficient $c_{\mu,\nu}^{\lambda}=1$ over the whole interval $[\ww,\nn]$, and a
 corresponding classification for Schur function products $s_\mu s_\nu$.

The Main Theorem, Theorem \ref{main},  gives a classification of the  multiplicity--free  skew Schur functions $s_A$ with full interval support \eqref{free}, up to a block of maximal width or
maximal depth, and up to a $\pi$--rotation and/or conjugation of the skew shape $A$. The classification consists of a list of seven different
configurations for $A$  with rather delicate conditions on relative part sizes of $A$ in Figure \ref{fig1}. Subsequent Corollary \ref{MainC} classify the multiplicity--free Schur functions products with full interval support. Together with the multiplicity--free (that is, all Littlewood-Richardson
 coefficients are $0$ or $1$)
 skew Schur functions classification due to Thomas-Yong \cite{thomas}, and Gutschwager \cite{gut},  and, therefore, the multiplicity--free Schur function products due to Stembridge \cite{Stem}, our proofs are based on  a procedure described in \cite{az}, here Algorithm 1 in Subsection \ref{subsec:alg}. The algorithm is made of several steps and
 produces all the LR fillings from the least to the most dominant one. Most of the  steps  in the algorithm are necessary to understand which skew shapes are prevented to attain the  full interval. A key tool towards our classification is the family of skew shapes in \eqref{f0} not attaining the full interval.

\subsection{ Schur support and  multiplicity-freeness } The  motivation to study full interval multiplicity--free skew shapes $A$ in \eqref{free}, comes  naturally  when one writes the expansion (\ref{motiv}), and, in particular, when one imposes in this expansion all the coefficients to be equal to 1. The dominance order ``$\preceq$'' on partitions has been used before in the study of Schur functions to prove that the monomial $x^\mu=x_1^{\mu_1} x_2^{\mu_2}\dots$ occurs in $s_\lambda$ if and only if $\mu\preceq \lambda$,
see~\cite{lam,LV}, and to deduce necessary conditions on the support of a skew Schur function $s_A$, namely, that the LR--filling contents of the skew shape $A$ vary respectively between those defined by the least and the most dominant LR--fillings of  $A$, see~\cite{az,Mn,zaballa}.
 The starting point of our  study has been a procedure described in \cite{az}, here Algorithm 1, which produces all the LR fillings from the least to the most dominant one.
Indeed the multiplicity--free phenomenon has been extensively  studied before by several authors.
In~\cite{Stem} the products of Schur functions that are multiplicity--free are completely classified,
i.e. products for which every coefficient in the resulting Schur function expansion is either $0$ or $1$.
This is done both for Schur functions in infinitely many variables and for Schur functions in finitely many variables. The latter being equivalent to a classification of all multiplicity--free tensor products of irreducible representations of ${ GL}_n$ or ${ SL}_n$, or in other words, it is completely determined when the outer products of characters of the symmetric groups have no multiplicity.
Afterwards, in~\cite{Bes} it is achieved the analogous classification for Schur $P$--functions, which solved a similar problem for (projective) outer products of spin characters of double covers of the symmetric groups, and finally in~\cite{SvW} it is solved the multiplicity--free problem for the expansion of Schur $P$--functions in terms of the Schur basis, which in turn yields criteria for when an irreducible spin character of the twisted symmetric groups in the product of a basic spin character with an irreducible character of the symmetric groups is $0$ or $1$. The characterisation of multiplicity--free skew Schur functions was solved in~\cite{gut,thomas}.
Furthermore, using a different combinatorial model, namely the hive model, in \cite{DTK} results similar to those investigated in~\cite{gut,Stem,thomas} are obtained.

Another problem that has received much attention, see for instance \cite{BBR,FFLP,Kir,KWvS,Mn,MnvW,MnvW1,LPP}, is to determine if the difference $s_A-s_B$ of two skew Schur functions is Schur-positive, i.e. if when expanded as a linear combination of Schur functions, all of the coefficients are nonnegative integers. A strong
necessary condition for Schur positivity of $s_A-s_B$ is that the support of $B$ is contained
in the support of $A$, and therefore, the ordering on skew shapes defined by support
containment is related to the Schur-positivity ordering.

 \subsection{Statement of the main results}  We may assume our skew shape $A$ without empty rows or empty columns. If ${\tilde A}$ is the skew Young diagram obtained from $A$ by deleting any empty row and any empty column, the corresponding skew Schur functions are equal
$s_{A}=s_{{\tilde A}}.$
 A  skew Schur function without empty rows or empty columns is said to be {\it basic} \cite{DTK}. This identity allows each skew Schur function to be expressed as a basic skew Schur function.  Schur functions are also invariant under $\pi$-rotation and conjugation. The classification in the Main Theorem of the full interval multiplicity--free skew shapes in \eqref{free} consists of a list of seven different configurations comprising restrictions on the relative part sizes as described in the Figure  \ref{fig1} below.
\begin{theorem}{\em(Main Theorem)}\label{main}
The basic skew Schur function $s_{\lambda/\mu}$ is multiplicity--free and its support is the whole Schur  interval $[{\bf w},{\bf n}]$ if and only if, up to a block of maximal width or maximal depth, and up to a $\pi$--rotation and/or conjugation,  at least one of the following is true:
\begin{itemize}
\item[$(i)$] $\lambda/\mu$ is a partition or a $\pi$-rotation of a partition;
\item[$(ii)$] $\lambda/\mu$ is a two column or a two row diagram ($A1$ configuration);
\item[$(iii)$] $\lambda/\mu$ is an $A2$ configuration;
\item[$(iv)$] $\lambda/\mu$ is an $A3$ configuration;
\item[$(v)$] $\lambda/\mu$ is an $A4$ configuration;
\item[$(vi)$] $\lambda/\mu$ is an $A6$ configuration;
\item[$(vii)$] $\lambda/\mu$ is an $A7$ configuration,
\end{itemize}
\noindent as described in Figure \ref{fig1}.
\end{theorem}
As a consequence of  Theorem  \ref{main} and the classification of multiplicity--free Schur function products \cite{Stem},   we get, in the corollary below, the
characterisation of the multiplicity--free  Schur function products that attain the full interval. It reduces to the product of two Schur functions whose indexing partitions  are precisely given by the configurations: partition, $A1$, and appropriate instances of $A_2$ and $A4$ in Figure \ref{fig1}, as detailed in $(c)$ and $(c')$ of the corollary.
\begin{corollary}\label{MainC}
The Schur function product $s_{\mu}s_{\nu}$ is multiplicity--free and its
support is the whole Schur interval if and only if at least one of the following is true:
\begin{itemize}
\item[$(a)$] $\mu$ or $\nu$ is the zero partition;
\item[$(b)$] $\mu$ and $\nu$ are both rows or both columns;
\item[$(c)$] $\mu=(1^x)$ is a one--column rectangle and  $\nu=(z,1^y)$
is a  hook such that either $z=2$ and $1\le x\leq y+1$, or $z\geq 3$ and $ x=1$ (or vice versa);
\item[$(c')$] $\mu=(x)$ is a one--row rectangle and  $\nu=(z,1^{y})$
is a  hook such that either $y=1$ and $1\le x\leq z$, or $y\geq 2$ and $ x=1$ (or vice versa).
\end{itemize}
\end{corollary}

 \begin{fig}\label{fig1}
 The seven full interval multiplicity--free skew shapes $(i)-(vii)$ in Theorem \ref{main}, up to a block of maximal width or
maximal depth, and up to a $\pi$--rotation and/or conjugation, with the inner shape coloured in blue:

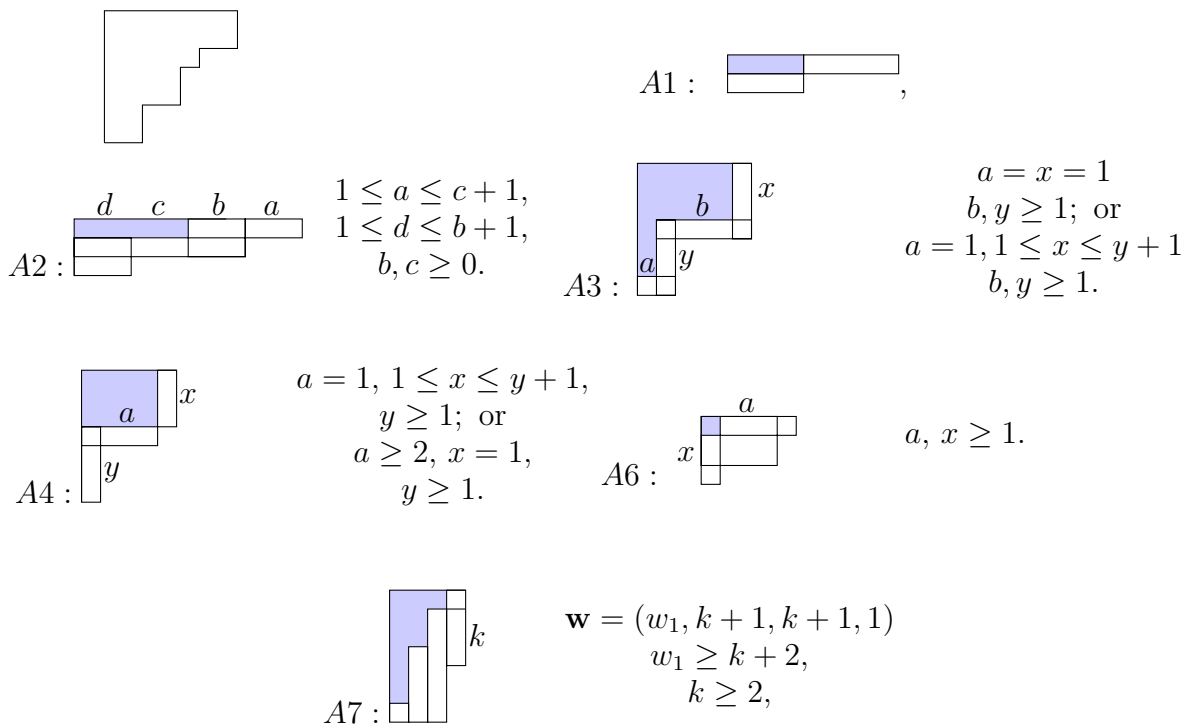
\begin{align*}\label{fullshapes}
\parbox{8.0cm}{$\quad \quad$ \begin{tikzpicture}
\draw (0,0) -- (0.5,0) -- (.5,.5) -- (1,.5) -- (1,1) -- (1.25,1) --
(1.25,1.25) -- (1.75,1.25) -- (1.75,1.75) -- (0,1.75) -- (0,0);
\end{tikzpicture}}\parbox{7.5cm}{$A1:\quad$\begin{tikzpicture}
\fill[color=blue!20] (0,.25) rectangle (1,.5);\draw (0,.25) rectangle (1,.5);
\draw (0,0) rectangle (1,.25);\draw (1,.25) rectangle (2.25,.5);
\end{tikzpicture},}\\
\parbox{4.3cm}{$A2:\;$\begin{tikzpicture}
\fill[color=blue!20] (0,.5) rectangle (1.5,.75);
\draw (0,0) rectangle (.75,.5)  (0,.25) rectangle (2.25,.5) (1.5,.25)
rectangle (2.25,.75) (2.25,.5) rectangle (3,.75);
\draw (0,0) -- (0,.75) -- (2,.75);
\node at (2.6,.9) {$a$};\node at (1.9,.95) {$b$};\node at (.4,.95)
{$d$};\node at (1.1,.9) {$c$};
\end{tikzpicture}}\parbox{3cm}{$\begin{matrix}1\le a\leq
c+1,\\ 1\le d\leq b+1,\\
b,c\ge 0.\\
\end{matrix}$}
\parbox{4.5cm}{$A3:\;$\begin{tikzpicture}
\fill[color=blue!20] (0,.25) rectangle (.25,1.75);
\fill[color=blue!20] (.25,1) rectangle (1.25,1.75);
\draw (0,0) rectangle (.5,.25)  (.25,0) rectangle (.5,1) (.25,.75) rectangle
(1.5,1) (1.25,.75) rectangle (1.5,1.75);
\draw (0,0) -- (0,1.75) -- (1.5,1.75);
\node at (1.7,1.4) {$x$};\node at (.65,.47) {$y$};\node at (0.13,.4)
{$a$};\node at (.85,1.2) {$b$};
\end{tikzpicture}}\parbox{4cm}{$\begin{matrix}a=x=1\\b,y\ge 1;\text{
or}\\ a=1, 1\le x\leq y+1\\ b,y\ge 1.\end{matrix}$}\\
\\
\parbox{3.7cm}{$A4:\;$\begin{tikzpicture}
\fill[color=blue!20] (0,1) rectangle (1,1.75);
\draw (0,0) rectangle (.25,1)  (0,.75) rectangle (1,1) (1,1) rectangle
(1.25,1.75);
\draw (0,0) -- (0,1.75) -- (1.25,1.75);
\node at (1.4,1.4) {$x$};\node at (.4,.4) {$y$};\node at (.6,1.15) {$a$};
\end{tikzpicture}}\parbox{4cm}{$\begin{matrix}a=1,\,
1\le x\leq y+1,\\
y\ge 1;\text{ or}\\ a\geq 2,\, x=1,\\y\ge 1.\end{matrix}$}
\parbox{4cm}{$A6:\;$\begin{tikzpicture}
\fill[color=blue!20] (0,1.2) rectangle (.25,1.45);
\draw (0,1.2) -- (0,1.45) -- (.25,1.45);
\draw (0,0.55) rectangle (.25,1.2) (.25,1.2) rectangle (1.25,1.45) (0,.8)
rectangle (1,1.45);\node at (.6,1.6) {$a$};\node at (-.2,.95) {$x$};
\end{tikzpicture}}\parbox{4cm}{$\begin{matrix}a,\,x\geq 1.\end{matrix}$}\\
\end{align*}
\begin{align*}
\parbox{3cm}{$A7:\;$\begin{tikzpicture}
\fill[color=blue!20] (0,.25) rectangle (.25,1.75);\fill[color=blue!20]
(0.25,1) rectangle (.5,1.75);
\fill[color=blue!20] (.5,1.5) rectangle (.75,1.75);
\draw (0,0) rectangle (.25,.25)  (.25,0) rectangle (.5,1) (.5,0) rectangle
(.75,1.5) (.75,1.75) rectangle (1,.75);
\draw (0,0) -- (0,1.75) -- (1,1.75);\draw (.75,1.5) -- (1,1.5);
\node at (1.15,1.15) {$k$};
\end{tikzpicture}}\parbox{4.5cm}{\quad$\begin{matrix}\ww=(w_1,k+1,k+1,1)\\
w_1\geq k+2,\\
k\geq 2,\end{matrix}$}\\
\end{align*}
where in the $A7$ configuration, $\ww$ is the partition formed by the column lengths of the skew diagram.
\end{fig}

\subsection{Organisation and contents}  The paper is organised in five sections. In the next,  which in turn is divided in four subsections, we give necessary definitions regarding partitions, skew shapes and operations on them;  the lattice of integer partitions with dominance order; Littlewood--Richardson tableaux using the notion of complete sequence of strings introduced in \cite{az}; Schur, skew Schur functions, and, following the presentation given in \cite{DTK}, the classification of multiplicity--free skew Schur functions due to Thomas and Yong \cite{thomas}, and Gutschwager \cite{gut},  and therefore the multiplicity--free Schur function products due to Stembridge \cite{Stem}. In section three the notion of Schur interval and support of a skew diagram and therefore of skew Schur function, considering the conjugate of the content of an LR tableau, are introduced. Algorithm \ref{Proc:alg} in \cite{az}, one of the main tools of this work, is introduced. Some other related results  in \cite{az,vW} are also recalled.

Section four is divided in two subsections. In Subsection 4.1, general skew shapes whose support does not achieve the Schur interval are deduced. We stress Lemma \ref{las} and Corollary \ref{TF3}, giving the family \eqref{f0} of skew shapes with non full interval support. They will be extensively used in the remaining of the paper   to prevent  full support.  On the other hand, we observe that horizontal (and vertical) strips have full support and that the support of a disconnected skew shape is equal to the Schur interval only if its components are ribbon shapes. (Although this is not sufficient.) Also,  in Proposition \ref{rib}, we conclude that the Schur function product $s_\mu s_\nu$
 has full interval support only if  $\mu$ and $\nu$ are either rows or columns, or one of the following holds: both hooks, a one-line rectangle and a hook or {\em vice-versa}.  In Subsection 4.2, the support of specific skew shapes, which includes all configurations $A_2,A_3,A_4,A_6$ and $A_7$, are analysed, and used in  the last
section devoted to the proof of our Main Theorem \ref{main}, and henceforth, Corollary \ref{MainC}.  We remark that the strategy  in the proof of Theorem \ref{main} follows closely the one used in  Lemma 7.1 \cite{DTK} as, roughly speaking, we have to {\em shrink} the multiplicity--free skew shapes in order to fit the full Schur interval.
 Finally, as a consequence of our analysis, we also include in the last section, Corollary \ref{cor:fullschur}, a classification of two Schur functions product with full interval Schur support, that is, with all LR coefficients positive.

\section{Preliminaries}

\subsection{Partitions and diagrams.}

Let $\mathbb N$ denote the set of non negative integers. A weakly decreasing sequence of positive integers $\lambda=(\lambda_1,\ldots,\lambda_{\ell(\lambda)})$,  whose sum is $n$ is said to be a
{\it partition} of $n$, denoted $\lambda\vdash n$. We say that $n$ is the {\em size} of $\lambda$, denoted  $|\lambda|$, and we call the $\lambda_i$ the {\em parts} of $\lambda$ and  $\ell(\lambda)$ the {\em length} of $\lambda$. It is convenient to set $\lambda_k=0$ for $k> \ell(\lambda)$.  We also let $0$ denote the partition with length $0$.
The set of all partitions of $n\in\mathbb N$ is denoted by $P_n$.
If $\lambda_i=\lambda_{i+1}=\cdots=\lambda_{i+j-1}=a$, we denote the sublist $\lambda_i,\ldots,\lambda_{i+j-1}$ by $a^j$, for $j> 0$.
We identify a partition $\lambda\vdash n$ with its {\em Young diagram}, in ``English convention", which we also denote by $\lambda$: containing $\lambda_i$ left justified boxes in the $i$th row, for $1\leq i\leq \ell(\lambda)$, and use the matrix--type coordinates to refer to the boxes.
For example, if $\lambda=(5,4,2)$, which we often abbreviate to $542$, the Young diagram is $$\tiny\young(\hfill\hfill\hfill\hfill\hfill,\hfill\hfill\hfill\hfill,\hfill\hfill).$$
A partition with at most
one part size is called a {\em rectangle}, and a partition with exactly two different part sizes is called a {\it fat hook}. A fat hook is said to be a {\em near rectangle} if it becomes a rectangle when one suppresses one row or one column, and   just a {\em hook} if it becomes a one row rectangle when we suppress a column.
If $\mu$ is another partition, we write $\lambda\supseteq\mu$ whenever $\mu$ is contained in $\lambda$  as Young diagrams, or, equivalently, $\mu_i\le \lambda_i$, for all $i\ge 1$. In this case, we define  the skew diagram $\lambda/\mu$ which is obtained from $\lambda$ by removing $\mu_i$ boxes from the $i$th row of $\lambda$, for $i=1,\ldots,\ell(\mu)$.  In particular,  $\lambda/0=\lambda$. The size of $\lambda/\mu$ is $|\lambda|-|\mu|$, denoted  $|\lambda/\mu|$.   A a skew--diagram is {\em connected} if, regarded as a union of solid squares, it has a connected interior; otherwise it is {\em disconnected}.
 A {\em ribbon shape} is a connected skew--diagram with no blocks of $2\times 2$ squares.  A skew--diagram forms a {\em vertical (respect. horizontal) strip} if it has no two boxes in the same row (respect. column). In particular, they are disconnected skew--diagrams or single rows or columns.
\begin{example}
    The skew diagram for  $\lambda/\mu=(4,4,2)/(2,1)$ is $$\tiny\young(::\hfill\hfill,:\hfill\hfill\hfill,\hfill\hfill),$$
     \noindent which is connected but it is not a ribbon shape. Instead, $\tiny\young(:::\hfill\hfill,:\hfill\hfill,\hfill\hfill)$ is  a disconnected skew--diagram with two components $\tiny\young(\hfill\hfill)$ and $\tiny\young(:\hfill\hfill,\hfill\hfill)$; and
     $\tiny\young(::\hfill\hfill\hfill,::\hfill,:\hfill\hfill,\hfill\hfill)$ is a ribbon shape.
   The following are respectively horizontal and vertical strips
    $\tiny\young(:::\hfill,::\hfill,\hfill\hfill),$   $\tiny\young(::::\hfill,:::\hfill,:::\hfill,::\hfill)$.
\end{example}

The $\pi$--{\em rotation} of a skew diagram $\lambda/\mu$, denoted $(\lambda/\mu)^{\pi}$, is obtained rotating $\lambda/\mu$ through $\pi$ radians. Denote by $\lambda'$ the partition obtained by transposing the diagram of  $\lambda$, called  {\em conjugate partition} of $\lambda$, and set $(\lambda/\mu)':=\lambda'/\mu'$.
If $\lambda\subseteq m^n$, then define its $m^n$--{\em complement} as $\lambda^*=(m^n)/\lambda$, where $\lambda^*_k=
m-\lambda_{n-k+1}$ for $k = 1, 2,\ldots, n$. In particular, $(\lambda^*)^\pi$ is a partition.

\begin{example}
If $\lambda/\mu=(4,4,2)/(2,1)$, then 
the $4^3$--complement of $\lambda$, the $\pi$--rotation, and the transposition are respectively
$$\lambda^*={\tiny\young(\hfill\hfill),\qquad(\lambda/\mu)^{\pi}=\tiny\young(::\hfill\hfill,\hfill\hfill\hfill,\hfill\hfill)},\quad\text{ and }\quad(\lambda/\mu)'=\tiny\young(::\hfill,:\hfill\hfill,\hfill\hfill,\hfill\hfill).$$
\end{example}

A partition $\lambda\subseteq m^n$
naturally defines a lattice path from the southwest to the northeast corner points of the
rectangle. Following Thomas and Yong \cite{thomas} we let the $m^n$--{\em shortness} of $\lambda$ to be
the length of the shortest straight line segment of the path of length $m + n$
from the southwest to northeast corner of $m^n$ that separates $\lambda$ from the $\pi$--rotation of
$\lambda^*$. For instance, if $\lambda=(4,4,2)$ then the path from the southwest to the northeast corner of the $4^3$ rectangle that borders $\lambda$ is $(2,1,2,2)$, and therefore the $4^3$--shortness of $\lambda$ is $1$.

The {\it sum} $\lambda+\mu$ of two partitions $\lambda$ and $\mu$, is the partition whose parts are equal to $\lambda_i+\mu_i$, with $i=1,\ldots,\max\{\ell(\lambda),\ell(\mu)\}$.
  Using conjugation, we define the {\it union} $\lambda\cup\mu:=(\lambda'+\mu')'$. Equivalently, $\lambda\cup\mu$ is obtained by taking all parts of $\lambda$ jointly with those of $\mu$ and rearranging all these parts in descending order.

\begin{example}
Let $\lambda=\tiny\young(\hfill\hfill\hfill,\hfill\hfill,\hfill)$ and $\mu=\tiny\young(\hfill\hfill,\hfill)$. Then,
$\lambda+\mu=\tiny\young(\hfill\hfill\hfill\hfill\hfill,\hfill\hfill\hfill,\hfill),\quad \lambda\cup\mu=
\tiny\young(\hfill\hfill\hfill,\hfill\hfill,\hfill\hfill,\hfill,\hfill).$

\end{example}

 Fix a positive integer $n$, and let $\lambda$ and $\mu$ be two partitions with length $\leq n$.
  The {\it product} $\lambda^{\pi}\bullet_n\mu$ of two partitions $\lambda$ and $\mu$ is defined as
 $$(\lambda_1+\mu_1,\ldots,\lambda_1+\mu_n)/(\lambda^*)^\pi,$$

\noindent where $\lambda^*=\lambda_1^n/\lambda$. (When it is clear from the context we shall avoid in the notation $\bullet_n$ the subindex $n$.)
   Graphically, place $\lambda^\pi$ in the southeast corner of the rectangle $\lambda_1\times n$, and place $\mu$ in the northwest corner of the rectangle $\mu_1\times n$. Then, form the rectangle $(\lambda_1+\mu_1)\times n$ by gluing together  the  rectangles $\lambda_1\times n$ and $\mu_1\times n$ in this order.   The outcome  diagram is a connected skew--diagram when $n<\ell(\lambda)+\ell(\mu)$ and a disconnected one otherwise.
  As  illustrated below with $\lambda=(3,2^2,0^2)$ and $\mu=(2,1^2,0^2)$, we obtain $\lambda^{\pi}\bullet_5\mu=(5,4^2,3^2)/(3^2,1^2)$,
$$\lambda^{\pi}\bullet\mu=
\tiny\young(:::\hfill\hfill,:::\hfill,:\hfill\hfill\hfill,:\hfill\hfill,\hfill\hfill\hfill).$$

\subsection{Dominance order on partitions}

\begin{definition} The {\em dominance order} on partitions $\lambda,\mu\vdash n$ is defined by setting $\lambda\preceq\mu$ if
$$\lambda_1+\cdots+\lambda_i\leq\mu_1+\cdots+\mu_i,$$
for $i=1,\ldots,\min\{\ell(\lambda),\ell(\mu)\}$.
 \end{definition}
 $(P_n,\preceq)$ is a lattice with maximum element $(n)$ and minimum element $(1^n)$, and is self dual under the map which sends each partition to its conjugate.
Graphically,  $\lambda\preceq\mu$ if and only if the diagram of $\lambda$ is obtained by ``lowering" at least one box in the diagram of $\mu$. Clearly $\lambda\preceq\mu$ if and only if $\mu'\preceq \lambda'$.
Moreover, $\mu$ {\em covers} $\lambda$, written as $\lambda\lhd\mu$, if and only if $\mu$ is obtained from $\lambda$ by lifting exactly one box in the diagram of $\lambda$ to the next available position such that the transfer must be from some $\lambda_{i+1}$ to $\lambda_i$ , or from $\lambda'_{i-1}$ to $\lambda'_i$. The interval $[\lambda, \mu]$ denotes the set of  all partitions $\nu$ such that
$\lambda \preceq \nu \preceq  \mu$. The chain $ \lambda = \lambda^0 \preceq \lambda^1\preceq\dots \preceq \lambda^k = \mu$, $k\ge 0$, is said to be   {\em saturated}
if $\lambda^{i}$ covers $ \lambda^{i-1}$, for $i = 1,\dots, k$ \cite{Bry, dominance}.

\subsection{ Littlewood--Richardson  tableaux.}\label{lrtableaux}

A \emph{semi--standard Young}
\emph{ta\-bleau} (SSYT) $T$ of shape $\lambda/\mu$ is a filling of the boxes in the diagram $\lambda/\mu$ with integers such that: $(i)$ the entries of each row weakly increase when read from left to right, and $(ii)$ the entries of each column strictly increase when read from top to bottom. The {\it reading word} $w$ of a SSYT $T$ is the word obtained by reading the entries of $T$ from right to left and top to bottom \cite{stanley}.
If, for all positive integers $i$ and $j$, the first
$j$ letters of $w$ includes at least as many $i'$s as $(i + 1)'$s, then we say that $w$
is a {\it lattice}. If $\alpha_i$ is the number of $i'$s appearing in $T$, and therefore in $w$, then the sequence $(\alpha_1,\alpha_2,\ldots)$ is called the {\it content} of $T$, and of $w$. Clearly, the content of a lattice word is a partition.
A SSYT $T$ whose word is a lattice is said to be  a Littlewood--Richardson tableau (LR tableau for short).
Given the partition $m=(m_1,\ldots,m_s)^\prime$ the set of all lattice words with content $m$ is equal to the set of all shuffles
of the $s$ words $12\cdots m_1$, $12\cdots m_2,\dots, 12\cdots m_s$.
Recall that a word $w$ is a shuffle of the words
$u$ and $v$ if $u$ and $v$  can be embedded as subwords of $w$ that occupy complementary
sets of positions within $w$. A shuffle $w$ of the words $u_1, u_2, \ldots , u_q$ is the empty word
for $q = 0$, the word $u_1$ for $q = 1$, and is, otherwise, a shuffle of $u_1$ with some shuffle of
the words $u_2,\ldots, u_q$ \cite{azma}.

\begin{example}\label{exemp}
    The following is a SSYT of shape $\lambda/\mu=(4,3,3,2,1)/(2,1,1)$, content $m=(4,2,2,1)=(4,3,1,1)^\prime$ and reading word $w=112132413$: $$\small\young(::11,:12,:23,14,3).$$
The reading word $w=112132413$ is a lattice word, and it is  a shuffle of the four words $1234,\,123,\;1$ and $1$. Therefore this SSYT is an LR tableau.
\end{example}

Taking into account  the shuffle property of a lattice word, we give another characterisation of a LR tableau that we shall rather use in this work.
This is based on \S 3 of~\cite{az}, especially Definitions 5 and 6 and Theorem 5, and we refer to it for proofs and further details.

\begin{definition}
Given a semi--standard tableau $T$, a sequence $S_k=\left(y_{1},y_{2},\ldots,y_{k}\right)$ of $k$ positive integers
 is a $k$--\emph{string} (or just string, for short, when there is no ambiguity) of $T$ if
$y_{1}<\cdots< y_{k}$ and the rightmost box in row $y_{j}$ is labeled with $j$; the corresponding strip, denoted by $st\left(S_{k}\right)$, is the union of all rightmost boxes in rows $y_{j}$, for all $j=1, \ldots ,k$.

We say that $S_{k}=\left(y_{1},y_{2},\ldots,y_{k}\right) \leq S_{t}=\left(z_{1},z_{2},\ldots,z_{t}\right)$
if $k \geq t$ and $y_{j} \leq z_{j}$ for all $j=1 , \ldots ,t$.

\vskip 1.5mm

We define in a recursive way $\left(S_{m_{1}}, S_{m_{2}}, \ldots  S_{m_{s}}\right)$ a \emph{complete sequence of strings} of the tableau $T$ having content $\left(m_1,m_2,\ldots,m_s\right)^{\prime}$ (note that we are writing the content in terms of its conjugate)
if $S_{m_{1}}$ is a string of $T$ and $\left(S_{m_{2}}, \ldots  S_{m_{s}}\right)$ is a complete sequence of strings of the tableau $T \setminus st\left(S_{m_{1}}\right)$ (when this set is not empty) having content $\left(m_2,\ldots,m_s\right)^{\prime}$.

In other word, $\left(S_{m_{1}}, S_{m_{2}}, \ldots  S_{m_{s}}\right)$ is a complete sequence of strings of $T$ if
$S_{m_{j}}$ is a string for $T\setminus \{\bigcup_{k=1}^{j-1} st\left(S_{m_{k}}\right)\}$ for all $j=1, \ldots ,s$.
\end{definition}

\begin{example}
$$T=\small\young(::11,:12,:23,14,3)$$ \noindent is a LR tableau with content $(4,2,2,1)=(4,3,1,1)'$
and it admits
$S_{4}=\left(1,2,3,4\right) \leq S_{3}=\left(1,3,5\right) \leq S_{1}=\left(2\right) \leq S_1=\left(4\right)$ as complete sequence of strings.

In fact, we have
\begin{eqnarray*}
T & = & \small\young(::11,:12,:23,14,3) \\
T \setminus st\left(S_{4}\right) & = & \small\young(::1,:1,:2,1,3) \\
\left(T \setminus st\left(S_{4}\right)\right) \setminus st\left(S_{3}\right) & = &\small \young(:1,1) \\
\left(\left(T \setminus st\left(S_{4}\right)\right) \setminus st\left(S_{3}\right)\right) \setminus st\left(S_{1}\right)
& = & \small\young(1).
\end{eqnarray*}

\end{example}

The following result holds, which is nothing but Theorem 5 in~\cite{az}.

\begin{proposition} \label{P:css}
A semi--standard tableau with content $m=(m_1,m_2,\ldots,m_s)^{\prime}$ is an LR tableau if and only if it has a complete sequence of strings $S_{m_1}, S_{m_2},\ldots, S_{m_s}$; and, in particular,
there is always one satisfying
$S_{m_1}\leq S_{m_2}\leq\cdots\leq S_{m_s}$.
\end{proposition}
\qed

\subsection{Schur functions, skew Schur functions and multiplicity--free classification.}

Let $\Lambda$ denote the ring of symmetric functions in the variables $x = (x_1,x_2,\ldots)$
over $\mathbb{Q}$, say. The Schur functions $s_{\lambda}$ form an orthonormal basis for $\Lambda$  \cite{stanley}, with respect to the Hall inner product, and may be defined in terms
of SSYT by
\begin{equation}\label{shur}
s_{\lambda}=\sum_{T}x^T\in\Lambda,\end{equation}
where the sum is over all SSYT of shape $\lambda$ and  $x^T$ denotes the monomial $$x_1^{\#1's\;in \;T}x_2^{\#2's\;in \;T}\cdots.$$
Replacing $\lambda$ by $\lambda/\mu$ in \eqref{shur} gives the definition
of the skew Schur function $s_{\lambda/\mu}\in\Lambda$, where now the sum is over all SSYT of shape
$\lambda/\mu$. For instance, the SSYT shown in the Example \ref{exemp} above contributes with the monomial $x_1^4x_2^2x_3^2x_4^1$ to $s_{43321/211}$.

The product of two Schur functions $s_{\mu}$ and $s_{\nu}$ can be written as a positive linear combination of Schur functions by the {\it Littlewood--Richardson
rule} which states
$$s_{\mu}s_{\nu}=\sum_{\lambda}c_{\mu\nu}^{\lambda}s_{\lambda},$$
where the {\it Littlewood--Richardson coefficient} $c_{\mu\nu}^{\lambda}$ is the number of LR tableaux with shape $\lambda/\mu$ and content $\nu$ \cite{lr}.
The Littlewood--Richardson coefficients can also be used to expand skew Schur functions $s_{\lambda/\mu}$ in terms of Schur functions:
$$s_{\lambda/\mu}=\sum_{\nu}c_{\mu\nu}^{\lambda}s_{\nu}.$$
If $c_{\mu\nu}^{\lambda}$ is 0 or 1  for all $\lambda$ (resp. all $\nu$), then we say that the product of Schur functions $s_{\mu}s_{\nu}$ (resp. the skew Schur function $s_{\lambda/\mu}$) is {\em multiplicity--free}.

The Littlewood--Richardson coefficients satisfy a number of symmetry properties \cite{stanley}, including:
\begin{equation}\label{shurP1}
c_{\mu\nu}^{\lambda}=c_{\nu\mu}^{\lambda}\quad\text{and}\quad c_{\mu\nu}^{\lambda}=c_{\mu'\nu'}^{\lambda'}.
\end{equation}
Moreover, we have
\begin{equation}\label{shurP2}
s_{\lambda}=s_{\lambda^{\pi}}\quad\text{and}\quad s_{\lambda/\mu}=s_{(\lambda/\mu)^{\pi}}.
\end{equation}
Another useful fact about skew Schur functions is that
$$s_{\lambda/\mu}=s_{{\tilde\lambda}/{\tilde\mu}},$$
where ${\tilde\lambda}/{\tilde\mu}$ is the skew Young diagram obtained from $\lambda/\mu$ by deleting any empty row and any empty column. A  skew Schur function without empty rows or empty columns is said to be {\it basic} \cite{DTK}. Therefore, the previous identity allows each skew Schur function to be expressed as a basic skew Schur function.

If $\lambda/\mu$ is not connected, and consists of two components $A$ and $B$,  and may themselves be either Young diagrams or skew Young diagrams, then the  combinatorial definition of (skew) Schur function \eqref{shur} gives (\cite{stanley,DTK})
$$s_{\lambda/\mu}=s_{A}s_{B}=s_Bs_A.$$

\begin{example}
If $\lambda/\mu=5522/421={\tiny\young(::::\hfill,::\hfill\hfill\hfill,:\hfill,\hfill\hfill)}$,  we have the disconnected components
$$
A={\tiny\young(:\hfill,\hfill\hfill)}\,\text{ and }\,B=\tiny\young(::\hfill,\hfill\hfill\hfill).$$
Therefore, $s_{5522/421}=s_{22/1}s_{33/2}$.
\end{example}

Any product $s_{A}s_{B}$ of skew Schur functions $s_{A}$ and $s_{B}$ is again a skew Schur function,
as  the figure below  makes evident,

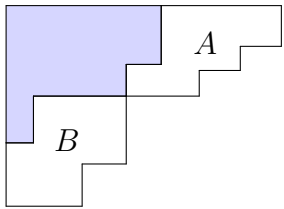
\begin{tikzpicture}[line cap=round,line join=round,>=triangle 45,x=1.0cm,y=1.0cm]
\clip(-2.3,0) rectangle (4.26,3);
\fill[color=blue!20,fill opacity=0.8] (0,0.92) -- (0.36,0.92) -- (0.36,1.54) -- (1.58,1.54) -- (1.58,1.96) -- (2.04,1.96) -- (2.04,2.74) -- (0,2.74) -- cycle;
\draw  (0,0.08)-- (1,0.08);
\draw  (1,0.08)-- (1,0.64);
\draw  (1,0.64)-- (1.58,0.64);
\draw  (1.58,0.64)-- (1.58,1.54);
\draw  (1.58,1.54)-- (0.36,1.54);
\draw  (0.36,1.54)-- (0.36,0.92);
\draw  (0.36,0.92)-- (0,0.92);
\draw  (0,0.92)-- (0,0.08);
\draw  (1.58,1.54)-- (2.54,1.54);
\draw  (2.54,1.54)-- (2.54,1.88);
\draw  (2.54,1.88)-- (3.08,1.88);
\draw  (3.08,1.88)-- (3.08,2.2);
\draw  (3.08,2.2)-- (3.62,2.2);
\draw  (3.62,2.2)-- (3.62,2.74);
\draw  (3.62,2.74)-- (2.04,2.74);
\draw  (2.04,2.74)-- (2.04,1.96);
\draw  (2.04,1.96)-- (1.58,1.96);
\draw  (1.58,1.96)-- (1.58,1.54);
\draw (2.62,2.24) node {$A$};
\draw (0.8,0.96) node {$B$};

\draw  (0,0.92)-- (0.36,0.92);
\draw  (0.36,0.92)-- (0.36,1.54);
\draw  (0.36,1.54)-- (1.58,1.54);
\draw  (1.58,1.54)-- (1.58,1.96);
\draw  (1.58,1.96)-- (2.04,1.96);
\draw  (2.04,1.96)-- (2.04,2.74);
\draw  (2.04,2.74)-- (0,2.74);
\draw  (0,2.74)-- (0,0.92);
\end{tikzpicture}.

In particular, a product of Schur functions $s_\mu s_\nu$ may be seen as a skew Schur function with $A=\mu$ and $B=\nu$ in the previous picture.

For the following characterisation of the basic multiplicity--free skew Schur function, jointly due to  Gutschwager and to
Thomas and Yong, we follow  \cite{DTK}.

\begin{theorem} [Gutschwager~\cite{gut}, Thomas and Yong~\cite{thomas}]\label{schurfree}
The basic skew Schur function
$s_{\lambda/\mu}$ is multiplicity--free if and only if at least one of the following is true:
\begin{itemize}
\item[$R0$] $\mu$ or $\lambda^*$ is the zero partition $0$;
\item[$R1$] $\mu$ or $\lambda^*$ is a rectangle of $m^n$--shortness 1;
\item[$R2$] $\mu$ is a rectangle of $m^n$--shortness 2 and $\lambda^*$ is a fat hook (or vice versa);
\item[$R3$] $\mu$ is a rectangle and $\lambda^*$ is a fat hook of $m^n$--shortness 1 (or vice versa);
\item[$R4$] $\mu$ and $\lambda^*$ are rectangles;
\end{itemize}
where $\lambda^*$ is the $m^n$--complement of $\lambda$ with $m = \lambda_1$ and $n =\lambda'_1$.
\end{theorem}

In particular, for partitions $\mu$ and $\nu$, the
product $s_{\mu}s_{\nu}$ of Schur functions is a skew Schur function, and we get the following characterisation of the multiplicity--free product of skew--Schur functions, due to Stembridge, as a corollary of the above theorem.

\begin{corollary}[Stembridge~\cite{Stem}]\label{stemp}
The Schur function product $s_{\mu}s_{\nu}$ is multiplicity--free if and only if at least one of the following is true:
\begin{itemize}
\item[$P0$] $\mu$ or $\nu$ is the zero partition $0$;
\item[$P1$] $\mu$ or $\nu$ is a one--line rectangle;
\item[$P2$] $\mu$ is a two--line rectangle and $\nu$ is a fat hook (or vice versa);
\item[$P3$] $\mu$ is a rectangle and $\nu$ is a near rectangle (or vice versa);
\item[$P4$] $\mu$ and $\nu$ are rectangles.
\end{itemize}
\end{corollary}


\section{The Schur interval}
\subsection{Skew Schur function support}
Given partitions $\mu\subseteq \lambda$, let $A$ denote the skew--diagram $\lambda/\mu$. We associate to the skew--diagram $A$ two partitions: $rows(A)$ obtained by sorting the row lengths of $A$ into weakly decreasing order, and similarly $cols(A)$  by sorting column lengths \cite{az,Mn}.
It is known
 that  $cols(A)\preceq rows(A)'$ \cite{LV,lam,zaballa,az,Mn}. For abbreviation, we write
 ${\ww}:=cols(A)$ and ${\nn}:=rows(A)'$. (When there is danger of confusion we write respectively $\ww(A)$ and $\nn(A)$.) If $A$ consists of two disconnected partitions $\phi$ and $\theta$ then $\ww=\phi\cup \theta$ and $\nn=\phi+\theta$.
\begin{definition} The interval  $[\bf{w},\bf{n}]=\{\nu\in P_{|A|}:\ww\preceq\nu\preceq\nn\}$ is called the {\em Schur  interval} of $A$.
\end{definition}

 The Schur interval of $A$ and $A^\pi$ is the same, and  due to the equivalence $\ww\preceq\nu\preceq \nn$ if and only if $\nn'\preceq\nu'\preceq\ww'$, the Schur interval of $A'$ is $[\nn',\ww']$.

 \label{deflambda1}Suppose $\nn=(n_1,\ldots,n_s)$. We may decompose $A=\lambda/\mu$ into a
sequence of $n_i$--vertical strips, for all $i$, as follows:

$1.$ First  consider the  $n_1$--vertical strip $V_1$ formed by the rightmost box of each row in
$A$, and let $A\setminus V_1$ be the skew diagram obtained by
removing that strip, described in partitions with $\lambda^1/\mu$.

$2.$ Repeat the previous step  with the diagram $A\setminus V_1$.

Let $V=(V_1,\ldots,V_s)$ denote the sequence of vertical strips
obtained by the previous process, called the {\em $V$--sequence of $A$}.  To construct the strip $V_i$  we  subtract the rightmost box of each row in $\lambda^{i-1}/\mu$, and to describe in partitions the new skew shape, we put  $A\setminus(V_1\cup V_2\cup\cdots\cup V_i)=\lambda^i/\mu$, for $i=1,\dots, s-1$, with  $\lambda^0:=\lambda$. This means that each entry of $\nn$ is obtained by successively pushing up some boxes in each entry of $\ww$:  $n_1$ is the number of non empty rows of $A$; $n_2$ is the number of rows of $A$ of length at least $2$, $\dots$, and $n_s$ is the number of rows of length $s$. Hence, $n_s$ is the number of rows of the strip $V_s=A\setminus(V_1\cup\cdots \cup V_{s-1})$, consisting of   the leftmost boxes  in each row of $A$ with longest length
  $s=\ell(\nn)$. Each vertical strip $V_i$ intersects the rows of $A$ with longest length $s$, and, therefore, $\ell(\ww)\ge\ell(\nn)$.

We shall denote  the minimum and the maximum of $\sup(\lambda^i/\mu)$ respectively by $\ww^i$ and $\nn^i$, for $i=1,\dots,s-1$.
If $\ww=(w_1,\ldots,w_r)$, then $r=\ell(\ww)$ is the number of non empty columns of $A$, and $w_1$ is the length of the longest column of $A$.

\begin{example}
    Consider the skew diagram $A=\tiny\young(::\hfill\hfill\hfill,:\hfill\hfill\hfill,\hfill\hfill\hfill,\hfill)$,
    with maximal filling $\nn=(4,3,3)$. The $V_1$ strip of $A$ is $\tiny\young(::::\hfill,:::\hfill,::\hfill,\hfill)$ and $A\setminus V_1=\tiny\young(::\hfill\hfill,:\hfill\hfill,\hfill\hfill)$. Next, we consider the $V_2$ strip $\tiny\young(::\hfill,:\hfill,\hfill)$ and $A\setminus(V_1\cup V_2)=\tiny\young(::\hfill,:\hfill,\hfill)$ coincides with the $V_3$ strip.
\end{example}

\begin{definition} Given the skew--diagram $A$,  we define the support
 $supp(A)$ of $A$, or of $s_A$, to be the set of those partitions $\nu'$ for which $s_\nu$
appears with nonzero coefficient when we expand $s_A$ in terms of Schur functions. Equivalently,
$$supp(A)=\{\nu': c_{\mu\nu}^{\lambda}>0\}.$$
\end{definition}

Notice that we have defined the support of $A$ in terms of the conjugate of the contents of the LR fillings of $A$.
 From \cite{LV,lam,zaballa,az,Mn} we know
 that $\nu\in supp(A)$  only if $cols(A)\preceq\nu\preceq rows(A)'$
 and therefore $\sup(A)\subseteq[\bf{w},\bf{n}]$.
 Thanks to the rotation symmetry \eqref{shurP2}, the support of a skew diagram $A$ equals the support of $(A)^{\pi}$. Also, by the conjugation symmetry of the Littlewood--Richardson coefficients \eqref{shurP1} and  the equivalence $\lambda\preceq \mu\Leftrightarrow\mu'\preceq\lambda'$, we know that $\nu\in  \sup(A)$  if and only if the  $\nu'\in\sup(A')$.

Moreover it is known that
$\bf{w}$ and $\bf{n}$ are in  $\sup(A)$ and
the coefficients $c_{\mu,\bf{w}'}^{\lambda}$ and
$c_{\mu,\bf{n}'}^{\lambda}$ are both equal to 1 (see \cite{az,Mn}).
The only LR tableau with shape $A$ and content $\bf{w}'$ is
obtained  by filling the boxes of each column, from top to bottom, with the
integers $1,2,\ldots$. To describe the only LR tableau with shape $A$ and
content $\bf{n}'$,
let $V=(V_1,\ldots,V_s)$ be the $V$--sequence of $A$. The LR tableau with shape
$A$ and content $\bf{n}'$ is obtained by filling each
vertical strip $V_i$ with the integers $1,\ldots,n_i$, $i=1,\ldots,s$.

Although $\ww'$ and $\nn'$ are respectively the most and the least dominant LR filling contents of  $A$, since $\ww$ and $\nn$ are the minimum and maximum of the
$supp(A)$, we will refer to the corresponding LR fillings of $A$ as the minimum and maximum ones. The reason for this terminology   comes from the fact that the lattice words with those contents $\ww'$ and $\nn'$ are respectively shuffles of the words $12\dots w_i$ and $12\dots n_i$, $i\ge 1$, and the partitions defined by their lengths  satisfy $\ww=(w_1,\ldots w_r)\preceq \nn=(n_1,\ldots,n_s)$.

\begin{example}\label{minimum}  The LR fillings of $A=5442211/331$   with the   least and most dominant conjugate contents respectively ${\ww}=cols(A)$ and ${\nn}=rows(A)'$, are
$$\begin{array}{ccccc}
\tiny\young(:::11,:::2,:113,12,23,3,4)&&&&
\tiny\young(:::11,:::2,:123,34,45,6,7)\cr
\end{array},$$

\noindent where $\ww=(4,3,3,1,1)\preceq\nn=(7,4,1)$. The lattice word of content $\ww'$ is a shuffle of the words $1234,\,123,\,123,1,1$ with lengths given by $\ww$; and the lattice word of content $\nn'$ is a shuffle of the words $1234567,\,1234,\,1$ with lengths given by $\nn$.
\end{example}

The example below shows that in general $\sup(A)\subsetneqq[\ww,\nn]$.
\begin{example}
(1) The support of  $A=\tiny\young(::\hfill\hfill,\hfill\hfill\hfill\hfill,\hfill\hfill\hfill)$ is only the set $\{\ww=3222,\nn=3321\}$ and therefore
$s_A=s_{441}+s_{432}$. In this case, $supp(A)=\{\ww=3222,\nn=3321\}=[\ww,\nn]$.

(2) The Schur interval of   $A=\tiny\young(:\hfill\hfill\hfill,\hfill,\hfill)$ is the chain $\ww=2111\prec 221\prec\nn=311$ while  $\sup(A)=\{21^3,31^2\}\subsetneqq[\ww,\nn]$.

\end{example}

\subsection{An algorithm to construct LR tableaux. }
\label{subsec:alg}
The next algorithm  provides a procedure to construct systematically all partitions in $\sup(\lambda/\mu)\cap[\bf{w},\bf{n}]$. Along the process, all LR tableaux of shape $\lambda/\mu$ are also exhibited.
We remark that the algorithm is essentially a rephrasing of Lemma 3 and 4 and Algorithm 4 in \S 4 of~\cite{az}, so we refer there for further details and proofs.

\begin{algorithm}
\begin{procedure} \label{Proc:alg}
\qquad \\
\emph{Input of the procedure:} an LR tableau $T$ of shape $\lambda/\mu$ and content $m=(m_1,m_2,\ldots,m_s)^{\prime}$ (therefore it admits a complete sequence of strings $\left(S_{m_{1}}, S_{m_{2}}, \ldots  S_{m_{s}}\right)$ since Proposition~\ref{P:css}).

\vskip 1.5mm

If for all $j=1, \ldots, s \; st\left(S_{m_{j}}\right)$  intersects all rows of the tableau
$T\setminus \{\bigcup_{k=1}^{j-1} st\left(S_{m_{k}}\right)\}$ then
\emph{Output of the procedure:} $T$ (i.e. the procedure does nothing) \\

\vskip 1.5mm

else

\vskip 1.5mm

\qquad Begin

\vskip 1.5mm

\qquad \qquad $t:=\min\{j=1, \ldots, s \text{ s. t. } st\left(S_{m_{j}}\right) \text{ does not achieve all rows of the}$

\qquad \qquad $\text{tableau } T\setminus \{\bigcup_{k=1}^{j-1} st\left(S_{m_{k}}\right)\}\}$.

\vskip 1.5mm

\qquad \qquad $t_{1}:=\min\{j \text{ s.t. the row }j \text{ in the tableau }
T\setminus \{\bigcup_{k=1}^{t-1} st\left(S_{m_{k}}\right)\} \text{ is not}$

\qquad \qquad $\text{ achieved by }
st\left(S_{m_{t}}\right)\}$.

\vskip 1.5mm

\qquad \qquad $X:=$ set of boxes made of the rightmost (with respect to the tableau

\qquad \qquad $T\setminus \{\bigcup_{k=1}^{t-1} st\left(S_{m_{k}}\right)\}$) box in rows
$\left(t_{1} \bigcup \{j > t_{1} \text{ s.t. } j \in S_{m_{t}}\}\right)$

\qquad \qquad (i.e. $X:=$ rightmost (with respect to the tableau
$T\setminus \{\bigcup_{k=1}^{t-1} st\left(S_{m_{k}}\right)\}$) box in

\qquad \qquad row $t_{1} \; \bigcup \left(st\left(S_{m_{t}}\right) \text{ below row } t_{1} \right)$).

\vskip 1.5mm

\qquad \qquad \emph{Output of the procedure:} $T_{1}:=$ tableau obtained by $T$ increasing by one

\qquad \qquad the filling of the set $X$.

\vskip 1.5mm

\qquad End

\end{procedure}

\vskip 2mm

\emph{Input of the algorithm:} $\mu\subseteq\lambda$.

\vskip 1.5mm

$\bf{n}:=\left(rows(\lambda/\mu)\right)^{\prime}$.

\vskip 1.5mm

$\bf{w}:=cols(\lambda/\mu)$.

\vskip 1.5mm

$n:=$ number of columns of $\lambda/\mu$.

\vskip 1.5mm

for $i=0,1, \ldots ,n-1$ do $(\lambda/\mu)^{n-i}:=$ the skew diagram defined by the $n-i,n-i+1,\ldots,n$ columns of $\lambda/\mu$.

\vskip 1.5mm

$T^{\{0\}}:=$ the LR tableau with shape $\lambda/\mu$ and content $\bf w'$.

\vskip 1.5mm

$T^{[n]}:=$ the LR tableau of shape $(\lambda/\mu)^n$.

\vskip 1.5mm


$i:=0$.

\vskip 1.5mm

Repeat

\vskip 1.5mm

\qquad Begin

\vskip 1.5mm

To each LR tableau $T\in T^{[n-i]}$, adjoin to the leftmost column of $T$
the $(n-i-1)$-th column of $T^{\{0\}}$ such that the LR tableau obtained is of shape $(\lambda/\mu)^{n-i-1}$.

\vskip 1.5mm

Apply the Procedure to construct all LR tableaux of shape $(\lambda/\mu)^{n-i-1}$ containing $T\in T^{[n-i]}$,
and denote this set by $T^{[n-i-1]}$.

\vskip 1.5mm

Add the remaining columns of $T^{\{0\}}$ to each LR tableau $T^{[n-i-1]}$, obtaining a set, denoted by
$T^{\{i+1\}}$, of LR tableaux of shape $\lambda/\mu$.

\vskip 1.5mm

\emph{Output of the algorithm:} set $T^{\{i+1\}}$.

\vskip 1.5mm

$i:=i+1$.

\vskip 1.5mm

\qquad End

\vskip 1.5mm

until $i=n$.
\end{algorithm}

This algorithm produces a sequence of sets of LR tableaux of shape $\lambda/\mu$
\begin{equation} T^{\{0\}}\subseteq T^{\{1\}}\subseteq T^{\{2\}}\subseteq\cdots\subseteq T^{\{n\}},\label{result:alg}
\end{equation}
\noindent such that if $G$ is  in  $T^{\{i\}}$ with conjugate content $\gamma$, and $B$ is in  $T^{\{i-1\}}$ with conjugate content $\beta$ then $\beta\preceq \gamma$, for all $i=0,\dots,n$.

\begin{example}
To make things clear we present here some instances of application of the Procedure~\ref{Proc:alg}.

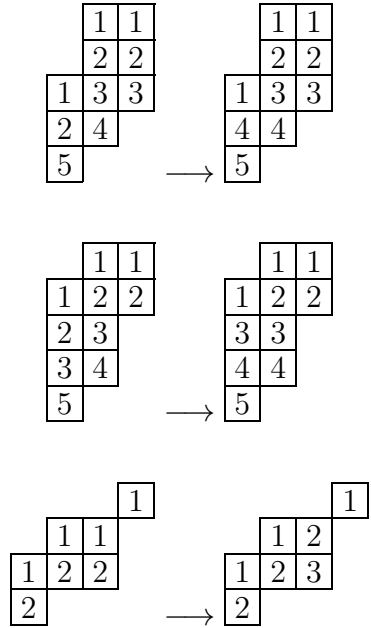
$$\young(:11,:22,133,24,5) \longrightarrow \young(:11,:22,133,44,5)$$

\vskip 1.5mm

$$\young(:11,122,23,34,5) \longrightarrow \young(:11,122,33,44,5)$$

\vskip 1.5mm

$$\young(:::1,:11,122,2) \longrightarrow \young(:::1,:12,123,2)$$

Note that in the first two instances the conjugate content of the output covers in the dominance order the conjugate content of the input, whereas in the third one it does not.
\end{example}

As an easy consequence of the algorithm above, we exhibit
 a chain in $supp(A)$, with respect to the dominance order, that goes from  $\ww=(w_1,\ldots,w_r)$ to $\nn=(n_1,\ldots,n_s)$. Start with the minimum LR filling of $A$, that is, the only filling of $A$ with content $\ww'$.  If one fills the vertical strip $V_1$ with $12\dots n_1$ then  $\lambda^1/\mu$ has the minimum LR filling, and one gets an LR filling of $A$ with conjugate content
  $$\sigma^1:=(n_1)\cup\ww^1,$$
 \noindent
 where $\ww^1$ is the minimum of the
$supp(\lambda^1/\mu)$.
 From Algorithm \ref{Proc:alg} one knows that $\ww\preceq\sigma^1\preceq\nn$. (It is worth noting that
 the partition $\sigma^1$ is obtained by subtraction from the
entries $2,\ldots, r$ of $\ww$, and adding those nonnegative quantities to the first entry.
Graphically, those operations correspond to lift  the rightmost box in some of the rows of $A$,
 to the first row, and thus  one has $\ww\preceq\sigma^1$.)
One may now repeat the
above argument with the skew diagram $\lambda^1/\mu$ and we are led  to a chain
 of partitions
\begin{equation}\label{chain} \ww\preceq\sigma^1\preceq\sigma^2\preceq\cdots\preceq\sigma^{s-2}\preceq\sigma^{s-1}=\nn \end{equation}
\noindent where $\sigma^i:=(n_1,\cdots,n_i)\cup \ww^i\in\sup(\lambda/\mu)$ and  $\ww^i$ the minimum of $\sup(\lambda^i/\mu)$,  for  $1\le i\le s-1$.
One has  $\ww^{i-1}\preceq (n_i)\cup \ww^i$, for  $1\le i\le s-1$, with $\ww^0:=\ww$.

The next example illustrates this construction.

\begin{example}
    Consider the skew diagram $A$ in Example \ref{minimum} and its sequence
    $V=(V_1, V_2, V_3)$ of vertical strips. Start with   the minimum LR filling of $A$ where $\ww=(4,3,3,1,1)$,  and apply Algorithm \ref{Proc:alg} to produce LR tableaux such that: the vertical strip $V_1$ is filled with the word $1234567$; and the vertical strips $V_1$ and $V_2$ are filled with $1234567$ and $ 1234$ (hence $V_3$ is filled with $1$) respectively
$$T^{\{0\}}=\tiny\young(:::11,:::2,:113,12,23,3,4),\qquad\quad \tiny\young(:::11,:::2,:113,14,25,6,7),\qquad\quad \tiny\young(:::11,:::2,:123,34,45,6,7).$$
  \noindent  The conjugate contents are respectively: $\ww$;  $\sigma^1=(72111)=(7)\cup (2111)$, with $\ww^1=2111$ the conjugate content of the minimum LR filling of $A\setminus V_1=\tiny\young(:::\hfill,:\hfill\hfill,\hfill,\hfill)$; and
 $\nn=(741)=(74)\cup (1)$, with $\ww^2=(1)$ the conjugate content of the minimum LR filling of $A\setminus (V_1\cup V_2)$ $=V_3$ $=\tiny\young(\hfill)$. One has
 $\ww\prec\sigma^1\prec\sigma^2=\nn.$
\end{example}


\begin{definition} The skew diagrams $A$ and $B$ are said to be equal up to a block of maximal depth if for some $x\in \mathbb N$, and $n\ge \ell(A),\ell (B)$,
$A=u^\pi\bullet_n\,v$ and $B=[u+(x^n)]^{\pi}\bullet_n\,v$, with $u$ and $v$ partitions. Similarly they are said to be equal up to a block of maximal width if $A=\left(u^{\pi}\bullet_n v\right)'$ and $B=\left([u+(x^n)]^{\pi}\bullet_n v\right)'$.
\end{definition}

\begin{example}
Let $u=(3,2,2,0)$,  $v=(3,1,1,0)$ and $n=4$. Then
$$A=u^{\pi}\bullet v={\tiny\young(:::\hfill\hfill\hfill,:\hfill\hfill\hfill,:\hfill\hfill\hfill,\hfill\hfill\hfill)}\text{ and }
B=[u+(2^4)]^{\pi}\bullet v=\tiny\young(:::\hfill\hfill\hfill\hfill\hfill,:\hfill\hfill\hfill\hfill\hfill,:\hfill\hfill\hfill\hfill
\hfill,\hfill\hfill\hfill\hfill\hfill),$$
\noindent are equal up to a block of maximal depth $(2^4)$.

\end{example}

\medskip

\begin{lemma}\label{bloco}
Let $A,\, B$ be  skew diagrams, and let $u,\,v$ be partitions with $\ell(u), \ell(v)\le n$. Let $x\in\mathbb N$.
\begin{enumerate}
 \item [(a)]  If $A=u^\pi\bullet_n\,v$ and $B=[u^\pi+(x^n)]\bullet_n\,v$, then $[\ww(B),\nn(B)]=[n^x\cup \ww(A),n^x\cup\nn(A)]$. If $A=\left(u^{\pi}\bullet_n v\right)'$ and $B=\left([u^{\pi}+(x^n)]\bullet_n v\right)'$, then $[\ww(B),\nn(B)]=[x^n+ \ww(A),x^n+\nn(A)]$.

 \item [(b)] ~{\em (\cite{az}, Lemma 5)}  If $A=u^{\pi}\bullet_n v$ and $B=u^\pi+(x^n)\bullet_n v$, then $c\in\sup(B)$ if and only if $c=b\cup(n^x)$ with $b\in\sup(A)$.
If $A=\left(u^{\pi}\bullet_n v\right)'$ and $B=\left(u^\pi+(x^n)\bullet_n v\right)'$, then  $c\in\sup(B)$ if and only if   $c=b+(x^n)$ with  $b\in\sup(A)$.
\end{enumerate}
\end{lemma}
 Note that this follows from the conjugation symmetry and  Algorithm \ref{Proc:alg}.
 In the application of this  algorithm, the rectangle $(x^n)$ of  length $n$ is always  filled with $x$  words $12\cdots n$. Therefore, if $c$ is an element of the support of $B$ then it has the form $b\cup (n^x)$ for some partition $b$ in the support of $A$.

\begin{example}
The following skew diagrams are equal  up to a block of maximal depth and maximal width
$$A=\tiny\young(::\hfill,\hfill\hfill),\quad B=\tiny\young(::\hfill,\hfill\hfill\hfill,\hfill\hfill\hfill,\hfill\hfill),\quad C=\tiny\young(::\hfill\hfill,\hfill\hfill\hfill\hfill,\hfill\hfill\hfill\hfill,\hfill\hfill\hfill),$$
$\ww(A)=111\lhd\nn(A)=21$; $\ww(B)=333=2^3+\ww(A)\lhd\nn(B)=2^3+\nn(A)=432$; $\ww(C)=4333=4^1\cup\ww(B)\lhd \nn(C)=4\cup \nn(B)=4432$.
\end{example}
\medskip
Next the skew shapes whose support has only one element, that is, $\ww=\nn$, are characterised, and, therefore, the skew Schur functions which are Schur functions. Note that, in \eqref{chain}, one has  $\sigma^i=\ww$, $1\le i\le s-1$, if and only if $A$ or $A^\pi$ is a partition.

\begin{proposition}\label{lemma}
     Let $A$ be a skew diagram and  let
   $u$, $v$  and $\nu$ be  partitions. Then,
\begin{enumerate}
\item [(a)] {\em (\cite{az}, Theorem 3, 16; \cite{bk}, Lemma 4.4)} $\ww=\nn=\nu$ if and only if $A=\nu$ or $A=\nu^{\pi}.$ In this case, $\sup(A)=\{\ww=\nn\}=[\ww,\nn]$.

\item [(b)]   {\em (\cite{vW})} $s_{A}=s_{\nu}$ if and only if $A=\nu$ or $A=\nu^{\pi}.$

\end{enumerate}
\end{proposition}


In the next proposition we characterise the skew diagrams $A$ whose support has only two elements, that is, $\sup(A)=\{\ww,\nn\}$, and, in particular, those whose Schur
interval has only two elements,
$[\ww,\nn]=\{\ww,\nn\}$. This was  shown in
\cite{az} by means of Algorithm 1.
Consider the skew diagrams: $F1=((a+1)^x,a)/(a^x)$, and  $\tilde F1=(a+1,a^x)/(a)$, $a,\,x\ge 1$:
\begin{align}\label{ini2}
\parbox{4cm}{$F1$\quad\begin{tikzpicture}
\fill[color=blue!20] (0,.25) rectangle (1.25,1.25);
\draw (0,0) rectangle (1.25,.25)  (1.25,.25) rectangle (1.5,1.25);
\draw (0,0) -- (0,1.25) -- (1.25,1.25);
\end{tikzpicture}}
\parbox{4cm}{$\widetilde{F}1$\quad\begin{tikzpicture}
\fill[color=blue!20] (-0.5,.75) rectangle (1,1);
\draw (-0.5,0) rectangle (1,.75)  (1,.75) rectangle (1.25,1);
\draw (-0.5,0) -- (-0.5,1) -- (1.25,1);
\end{tikzpicture}}.
\end{align}

\begin{proposition}\label{prop12}\label{prop13}\label{prop1}
 {\em (\cite{az}, Theorem 16)}
    Let $A$ be a skew diagram with
$\ww\precneqq\nn$. Then, $\sup(A)=\{\ww,\nn\}$ if and only
if, up to a $\pi$--rotation/ or conjugation and up to a block of maximal width or
maximal depth, $A$ either is an $F1$ or an $\widetilde{F}1$
configuration. In particular, if $A$ is a disconnected two column (row) diagram,  with
one connected component  a single box, then one has $\ww\vartriangleleft\nn$ and
  $\sup(A)= [\ww,\nn]=\{\ww,\nn\}$.
\end{proposition}

\begin{remark} \label{cr} Note that, when applying Algorithm 1 to $F1$, we find that the only string in the minimum LR filling of $F1$ that we can stretch is the string of length $x$, which can only be stretched in one way, which gives rise to the maximum LR filling.  Therefore, the   support of $F1$ is formed only by $\ww=(x,1^a)\preceq\nn=(x+1,1^{a-1})$. When $a,\,x\ge 2$, the partition $\nn$ does not cover $\ww$, for instance $\xi=(x,2,1^{a-2})\in[\ww,\nn]$ and, therefore, the support is not the full interval.
The proof is similar for $\tilde F_1$, considering its $\pi$--rotation.
\end{remark}


\begin{example}
    Proposition \ref{prop1} can be used together with Lemma \ref{bloco} to show
that the support of the skew diagram
$\lambda/\mu=(5,3^3)/(1^2)=\tiny\young(:\hfill\hfill\hfill\hfill,:\hfill\hfill,\hfill\hfill\hfill,\hfill\hfill\hfill)$
is not the entire Schur interval. By Lemma \ref{bloco}, it is enough
to consider the support of the simple skew diagram
$\alpha/\beta=\tiny\young(:\hfill\hfill,\hfill,\hfill)$ obtained from
$\lambda/\mu$ by removing the block $(2^4)$. Since the resulting
diagram is the $\pi$--rotation of an $F1$ configuration with $a,x\ge 2$,
the support of $\lambda/\mu$ is strictly contained in its Schur
interval.
\end{example}

\section{Recognition of  full and non full interval supports}

 In the rest of the paper, the general philosophy of the application of  Algorithm \ref{Proc:alg}, to a skew diagram $A$, consists in the prolongation of its strings, starting with the minimum LR filling of $A$, in any  possible way.
We remark that in ~\cite{az} it is characterised when in Algorithm ~\ref{Proc:alg}, the partition corresponding to the content of the output covers, in the dominance order, the partition corresponding to the content of the input. Thus, when applying a step of the algorithm we can check whether the partition, corresponding to the new content, covers the preceding one. If not, then we have a suspicious interval that may contain a partition not in the  support of $A$.
\subsection{Bad configurations}
We start this section with the analysis of some particular configurations of boxes such that their  appearance in a skew diagram $A$ implies that  $A$ has non full support, $\sup(A)\subsetneqq[{\bf w},{\bf n}]$.


\begin{lemma}\label{induction} If $A$ is a skew--diagram and the support of $A\setminus V_1$ is not the entire Schur interval, then neither is  the support of $A$.

\end{lemma}
\begin{proof} Let $[\ww^1,\nn^1]$ be the Schur interval of $A\setminus V_1$ and $\ww^1\preceq\xi\preceq\nn^1$ such that $\xi\notin\sup(A\setminus V_1)$. Then $\ww\preceq (n_1)\cup \xi\preceq \nn$ and $(n_1)\cup \xi\notin \sup(A)$ since the only way to put the string $n_1\cdots 21$ in $A$ is to fill the strip $V_1$ with it and what remains is $A\setminus V_1$.
\end{proof}

We observe that if $\sup(A\setminus V_1)$ attains the Schur interval this does not mean that the same happens to $A$, as one can see in the next example.

\begin{lemma}\label{disconnected}
Let $A$ be a skew diagram with two or more connected components. If there is a component containing a two by two block of boxes, then  the support of $A$ is not the entire Schur interval.
\end{lemma}
\begin{proof}
Let  $\nn=(n_1,\ldots,n_s)$.  Recall that $\nn^{i-1}=(n_i,\ldots,n_s)$  is the maximum of $\sup(A\setminus \bigcup_{k=1}^{i-1} V_k)$, $i=2,\cdots s$, and  that $n_i$ is the number of rows of $A\setminus \bigcup_{k=1}^{i-1} V_k$, for all $i$.
Since there is a 2 by 2 block in  one of the connected components of $A$, there must exist a column in $A\setminus V_1$ whose length is at least $ 2$. Let  $\ww^1=(\overline{w}_1,\ldots,\overline{w}_{\ell},1^q)$, with $\overline{w}_{\ell}\geq 2$  for some $\ell\geq 1$ and $q\geq 0$. Clearly $(n_2,\ldots,n_s) \succcurlyeq(\overline{w}_1,\ldots,\overline{w}_{\ell}-1,1^{q+1})$, and
from \eqref{chain}  the partition
$$\sigma^1=(n_1,\overline{w}_1,\ldots,\overline{w}_{\ell},1^q)\in[\ww,\nn].$$
Note that $\ell(\ww)\ge \ell(\ww^1)+o$ where $o$ is the number of components of $A$. Since $A$ has at least two components, one has $\ell(\ww)\geq\ell(\sigma)+2$. Then, the partition
$$\xi:=(n_1,\overline{w}_1,\ldots,\overline{w}_{\ell}-1,1^{q+1})$$
clearly satisfy $\ww\preceq\xi\preceq\sigma$. Moreover, since $(\overline{w}_1,\ldots,\overline{w}_{\ell}-1,1^{q+1})\precneqq\ww^1$ it follows that
$(\overline{w}_1,\ldots,\overline{w}_{\ell}-1,1^{q+1})\notin\sup(A\setminus V_1)$, and therefore we conclude that $\xi\notin\sup(A)$.
\end{proof}

\begin{example}
The support of the skew diagram $A=\tiny\young(:::\hfill,:\hfill\hfill,:\hfill\hfill)$ is not the entire Schur interval, since it has two connected components, and one of them has a 2 by 2 block. We may follow the proof of the previous lemma to get a partition in the Schur interval that does not belong to the support of $A$. Note that $\ww=(2,2,1)$, $\nn=(3,2)$, $\ww^1=(2)=\nn^1$, $\ww\preceq\sigma^1=(3,2)=\nn$, and
$[\ww,\nn]=\{\ww=221,\xi=311,\nn=32\}$ with $\xi\notin\sup(A)=\{\ww,\nn\}$.
Note also that $A\setminus V_1=V_2$ and $V_2$ is a column of $A$.
\end{example}

\begin{corollary} \label{dis} If $A$ is a skew diagram with two or more components and the support of $A$ is the whole Schur interval $[\ww,\nn]$, then the components of $A$ are ribbon shapes.
\end{corollary}

\begin{corollary}\label{same}
Let $A$ be a skew diagram
such that $\ell(\ww)>\ell(\nn)=s$ (equivalently, it has no block of maximal width),  and the strip $V_s$ is a column of
 $A$ of length greater than, or equal to 2. Then, the support of $A$ is not $[{\ww},{\nn}]$.
\end{corollary}

\begin{proof} Since $\ell(\ww)>\ell(\nn)=s$, $A$ is not a partition, and since $V_s$ is a column,  $A$ has precisely $|V_s|\ge 2$ rows of length $s\ge 2$ such that they form a rectangle, therefore, containing a $2\times 2$ block of boxes. If $A$ is disconnected it is done. Otherwise, as all vertical strips $V_i$ $1\le i\le s$, transverse that rectangle, we may delete  the vertical strips $V_1,\ldots,V_k$, for some $1\le k<s-1$, until getting a disconnected skew diagram. At this point, we are in conditions of Lemma \ref{disconnected}, and, from Lemma \ref{induction}, we are done.
\end{proof}

\begin{example}
For instance, it follows from Corollary \ref{same} that the supports of
the skew diagrams
$$\begin{array}{cccccccccccccc}A=\tiny\young(
::\hfill\hfill,:\hfill\hfill,\hfill\hfill\hfill,\hfill\hfill\hfill);&&
B=\tiny\young(
::\hfill,:\hfill,\hfill\hfill,\hfill\hfill);&&
C=\tiny\young(
:::\hfill,::\hfill,:\hfill\hfill,:\hfill\hfill,\hfill);&&
 D=\tiny\young(
:::\hfill,::\hfill\hfill,\hfill\hfill\hfill,\hfill\hfill\hfill)
\end{array}$$
\noindent are strictly contained in the Schur interval $[\ww,\nn]$.  In the first, for instance, one has
$\ww=4321$ and $\nn=442$, where $V_3$ is a column of $A$ of length two, and $\ell(\ww)=\ell(\nn)+1$.
The Schur interval is $[\ww,\nn]=\{\ww=4321;\,4411;\,433;\,\nn=442\}$ and the partition $\xi=4411\notin\sup(A)=\{\ww=4321,\,433;\,\nn=442\}$.
Therefore $s_A=s_\ww+s_{433}+s_\nn$. In the last, one has
$\xi=4311\in[\ww=32^3,\nn=432]\setminus\sup(D)$.
\end{example}

\begin{lemma}\label{las}
Let $A$ be a connected skew diagram  such that
$$\ww=(w_1,\ldots,w_r)\preceq\sigma^1=(n_1)\cup\ww^1=(n_1,\overline{w}_2,\ldots,\overline{w}_{\ell},w_{\ell+1},\ldots,w_r)\preceq\nn=(n_1,\ldots,n_s)$$
for some $3\leq \ell\leq r$ such that $\overline{w}_k\leq w_k$ for $k=1,\ldots,\ell$
and $0<\overline{w}_{\ell}<w_{\ell}$. Moreover, assume the existence of two integers
$2\leq i<j\leq\ell$ such that  $\overline{w}_i\geq \overline{w}_j+2$ and $w_j> \overline{w}_j$.  Then
the support of $A$ is not the entire Schur interval.
\end{lemma}
\begin{proof}
Consider the partition $\xi$ obtained from $\sigma$ by
replacing the entries $\overline{w}_i$ and $\overline{w}_j$ by $\overline{w}_i-1$ and $\overline{w}_j+1$,
respectively. It is clear that $\xi\preceq
\sigma\preceq\nn$. Note also that while $\xi$ is obtained from $\sigma$ by lowering
one box, from one row of length $\overline{w}_i$ to one of length $\overline{w}_j$,
$\ww$ is obtained from $\sigma$ by lowering $n_1-w_1=w_2+\cdots+w_l-(\overline{w}_2+\cdots+\overline{w}_l\ge 1)$ boxes from the first
row to some rows in  which is included that one of length $\overline{w}_j$, since
$w_j>\overline{w}_j$. Thus, $\ww$ can be obtained from $\xi$ by lowering $k=n_1-w_1$ boxes, in particular, $w_i-\overline{w}_i+1$ boxes to row $i$ and $w_j-\overline{w}_j+1$ to row $i$.
Thus $\ww\preceq \xi\preceq \sigma$.

Moreover, $\xi^1\prec\ww^1$, where $\xi^1$ denotes
the partition obtained from $\xi$ removing the first entry. Thus,
$\xi^1\notin\sup(A\setminus V_1)$.
From Lemma \ref{induction}, we conclude that $\xi\notin\sup(A)$.
\end{proof}


\begin{example}
    We can use the lemma above to conclude that the support of the connected skew diagram $A=\tiny\young(::\hfill,:\hfill\hfill,:\hfill\hfill,:\hfill\hfill,\hfill\hfill,\hfill)$ is not the entire Schur interval, since
    $\ww=(4,4,2)\preceq\sigma^1=(6,3,1)$, and the last two entries of $\sigma^1$ differ in two unities, and they are strictly smaller than the correspondent entries of $\ww$.
\end{example}


As a consequence of the last lemma we describe below a large group of skew diagrams whose support is strictly contained in the Schur interval.
Let $F0$ be a skew diagram having two columns  with the same length, and starting and ending on the same rows, say $x$ and $y$, and such that  all columns,  to the right of those two equal columns, end at least two rows above row $y$, with at least one of these columns starting at least one row above row $x$, as illustrated by:
\begin{equation}\label{f0} F0\;\begin{tikzpicture}
\fill[color=blue!20] (0,1) rectangle (.5,3.25);
\fill[color=blue!20] (.5,1.5) rectangle (1.25,3.25);
\fill[color=blue!20] (1.25,2) rectangle (1.75,3.25);
\fill[color=blue!20] (1.75,2.25) rectangle (2.5,3.25);
\fill[color=blue!20] (2.5,2.75) rectangle (3,3.25);
\draw (0,0.5) -- (0,1) -- (.5,1) -- (.5,1.5) -- (1.25,1.5) -- (1.25,2) -- (1.75,2) -- (1.75,1) -- (1.25,1) -- (1.25,.5) -- (.75,.5) -- (.75,0.5) -- (0,0.5);
\draw (1.75,1) rectangle (2,2.25)  (2,1) rectangle (2.25,2.25);
\draw (2.25,1.5) -- (2.75,1.5) -- (2.75,1.8) -- (3.25,1.8) -- (3.25,2.4) -- (3.75,2.4) -- (3.75,3.25) -- (3,3.25) -- (3,2.75) -- (2.5,2.75) -- (2.5,2.25) -- (2.25,2.25);
\draw (0,0.5) -- (0,3.25) -- (3,3.25);
\draw[blue] (1.75,2.25) -- (4.25,2.25);
\node at (2.6,1.25) {$\geq 2$};\node at (4.2,2.75) {$\geq 1$};
\end{tikzpicture}.\end{equation}
Denote by $F0', F0^{\pi}$ and $F0^{\pi'}$ the skew diagrams which are, respectively, the conjugate, the $\pi$--rotation, and the conjugate of the $\pi$--rotation of an $F0$ skew diagram.

\begin{corollary}\label{TF3}
The support of the skew diagrams  $F0$, $F0'$, $F0^{\pi}$, or  $F0^{\pi'}$ is strictly contained in the Schur interval.
\end{corollary}
\begin{proof}
Thanks to the conjugate symmetry in  \eqref{shurP1} and to the $\pi$--rotation symmetry \eqref{shurP2}, it is enough to consider the $F0$ configuration. Denote by $a+b$ the length of column $W_i$, which is also the length of column $W_{i+1}$, of $F0$,  where $b\geq 0$ is the number of boxes that column $W_{i+1}$ shares with the column to its right, and $a\geq 2$ is the number of boxes of $W_{i+1}$ with no right neighbour.

Consider $F0\setminus V_1$  and $\ww^1=(\overline{w}_2,\ldots,\overline{w}_{\ell},w_{\ell+1},\ldots,w_q)$ with $\overline{w}_f=a+b$ and $\overline{w}_g=b$ satisfying $\overline{w}_f\geq \overline{w}_g+2$, for some integers $2\leq f<g\leq \ell\leq q$. By Lemma \ref{las}, it follows that $\sup(F0)$ is strictly contained in the Schur interval.
\end{proof}

\begin{example} (a)
To illustrate the previous corollary, consider the skew diagram $$A=\tiny\young(::::\hfill\hfill,::\hfill\hfill\hfill,:\hfill\hfill\hfill,\hfill\hfill\hfill\hfill,\hfill),$$
\noindent  with  $\ww=(3,3,2,2,2,1)$ and $\nn=(5,4,3,1)$. Clearly,  $A$ is a $F0$ configuration since the third and fourth columns have the same length and they start in the same row,  to its right all columns end two rows above the last row of these columns, and there are columns that start one row above the first row of these columns. Thus, by the previous corollary, the support of $A$ is not the entire Schur interval. Moreover, following the proof of Lemma \ref{las}, we construct the partition  ${\xi}=(5,2^3,1^2)$ which  belongs to the interval $[\ww,\nn]$ but is not an element of $\sup(A)$.

(b)  Consider now the skew--diagram
 $$B=\tiny\young(::::\hfill\hfill,:\hfill\hfill\hfill\hfill\hfill,\hfill\hfill\hfill\hfill\hfill,\hfill\hfill\hfill,\hfill\hfill\hfill).$$

\noindent One has $\ww=443322\preceq \xi= 533322\preceq \sigma=543222$ and $\xi\notin \sup(B)$.

\end{example}
We  now are in conditions to conclude from Remark \ref{cr} and Corollary \ref{dis}

\begin{proposition}\label{rib} If $A$ is a skew diagram with two or more components and the support of $A$ is the whole Schur interval, then the components of $A$ are ribbon shapes. In particular, the Schur function product $s_\mu s_\nu$
  has Littlewood-Richardson coefficients always positive over the full interval only if  $\mu$ and $\nu$ are either rows or columns, or one of the following holds: both hooks, a one-line rectangle and a hook or {\em vice-versa}.

\end{proposition}


\subsection{Recognition of full  configurations}
Next we examine some particular configurations  which are  needed  to characterise the multiplicity--free skew Schur functions that achieve the full interval. We start with the two row (column) skew diagram. (The one row (column) case is a partition and, therefore, the support is $[\ww,\nn]=\{\ww=\nn\}$.)

\begin{proposition}\label{2col}
{\em (~\cite{az}, Theorem 15)}
If $\lambda/\mu$ has exactly two rows, then $\sup(\lambda/\mu)=[\ww,\nn]$ and it is a saturated chain.
More generally, if $\lambda/\mu$ is a skew diagram with one of the configurations \eqref{ini}, then the support of $\lambda/\mu$ is the entire interval $[\ww,\nn]$, and it is a saturated chain.
\end{proposition}

Let $\lambda/\mu=(2^a,1^b)/(1^c)$ be a two column skew diagram.
Lemma \ref{bloco} shows that the support of a skew diagram $A=((x+2)^{a+b},(x+1)^c)/(1^a)$,  obtained by inserting a block $(x^{a+b+c})$ of maximal length between the two columns of $\lambda/\mu$, or a skew diagram $B=((n+2)^{a+b},1^c)/((x+1)^a)$ obtained adding the partition $(n^{a+b})$ to  $\lambda$ and  $(n^{a})$ to $\mu$, is again equal to its entire Schur interval.  This construction, together with its conjugates, yields four cases, whose schematic representations  are shown below.  These diagrams have been arranged so that the ones in the right column are the conjugates of those in the left column.

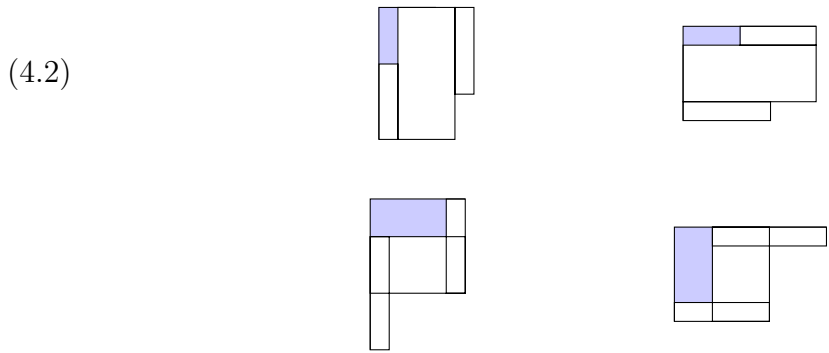
\begin{align}\label{ini}
\parbox{4cm}{\qquad\begin{tikzpicture}
\fill[color=blue!20] (0,1) rectangle (.25,1.75);
\draw (0,0) rectangle (.25,1)  (.25,0) rectangle (1,1.75) (1,1.75) rectangle (1.25,.6);
\draw (0,0) -- (0,1.75) -- (.75,1.75);
\end{tikzpicture}}
\parbox{4cm}{\qquad\begin{tikzpicture}
\fill[color=blue!20] (0,1) rectangle (.75,1.25);
\draw (0,0) rectangle (1.15,.25)  (0,.25) rectangle (1.75,1) (1.75,1) rectangle (.75,1.25);
\draw (0,0) -- (0,1.25) -- (1.75,1.25);
\end{tikzpicture}}\\\nonumber
\\
\parbox{4cm}{\qquad\begin{tikzpicture}
\fill[color=blue!20] (0,1.5) rectangle (1,2);
\draw (0,0) rectangle (0.25,1.5)  (0,.75) rectangle (1.25,1.5) (1,.75) rectangle (1.25,2);
\draw (0,0) -- (0,2) -- (1.25,2);
\end{tikzpicture}}
\parbox{4cm}{\qquad\begin{tikzpicture}
\fill[color=blue!20] (0,.25) rectangle (.5,1.25);
\draw (0,0) rectangle (1.25,.25)  (.5,0) rectangle (1.25,1.25) (.5,1) rectangle (2,1.25);
\draw (0,0) -- (0,1.25) -- (2,1.25);
\end{tikzpicture}}.\nonumber
\end{align}

{\em  A two column or a two row disconnected skew diagram  is called an $A1$ configuration},

\begin{align}
\parbox{4cm}{\qquad
\begin{tikzpicture}
\fill[color=blue!20] (0,.25) rectangle (1,.5);\draw (0,.25) rectangle (1,.5);
\draw (0,0) rectangle (1,.25);\draw (1,.25) rectangle (2.25,.5);
\end{tikzpicture}.}
\end{align}

Skew Schur functions whose  shapes are strips made either of columns or rows always attain the full interval.
Let $e_n$ be the elementary symmetric function and $h_n$ the complete homogeneous symmetric function of degree $n$. Then
$s_{(1^n)}=\sum_{{i_1}<\ldots<i_n} x_{i_1}\ldots x_{i_n}=e_n,\quad s_n=\sum_{{i_1}\le \ldots\le i_n} x_{i_1}\ldots x_{i_n}=h_n.$
Let $\mu=(\mu_1,\ldots,\mu_l)$ be a partition  and $A=(1^{\mu_1})\oplus\cdots \oplus(1^{\mu_l})$. The Schur interval of $A$ is $[\mu,({|\mu|})]$.
It is known \cite{ful,stanley} that
$s_A=e_\mu:=e_{\mu_1}e_{\mu_2}\ldots e_{\mu_l}=\sum_{\lambda}K_{\lambda,\mu}s_{\lambda'}.$
As the Kostka number $K_{\lambda,\mu}\neq 0$ if and only if $\mu\preceq \lambda$, equivalently, $\lambda\in[\mu,({|\mu|})]$. Then
$s_A=\sum_{\lambda\in[\mu,(1^{|\mu|})]}K_{\lambda,\mu}s_{\lambda'},$ and $\sup(A)$ equals its Schur interval. Thus $\sup(A')$ or the support of $h_\mu$
is its Schur interval $[(1^{|\mu|}), (\mu')]$. Therefore, {\em if
 $A$ is either a direct sum of columns or of rows,
 the support of $A$ equals its Schur interval.}

Consider now skew diagrams $\lambda/\mu$  with three rows (columns)  where $\mu=(d+c)$ is a one row (column) rectangle and $\lambda^*=(a+b+c,a)$ is a fat hook, or vice versa, for some integers $a,d\geq 1$ and $b,c\geq 0$. There are four cases, as illustrated by:

\begin{align}\label{a3}
\parbox{6cm}{$F2$\quad\begin{tikzpicture}
\fill[color=blue!20] (0,.5) rectangle (1.5,.75);
\draw (0,0) rectangle (.75,.5)  (0,.25) rectangle (2.25,.5) (1.5,.25) rectangle (2.25,.75) (2.25,.5) rectangle (3,.75);
\draw (0,0) -- (0,.75) -- (2,.75);
\node at (2.6,.9) {$a$};\node at (1.9,.95) {$b$};\node at (.4,.95) {$d$};\node at (1.1,.9) {$c$};
\end{tikzpicture}}
\parbox{6cm}{$F2^{\pi}$\quad\begin{tikzpicture}
\fill[color=blue!20] (0,.25) rectangle (.75,.75);
\fill[color=blue!20] (.75,.5) rectangle (2.25,.75);
\draw (0,0) rectangle (1.5,.25)  (.75,.25) rectangle (3,.5) (2.25,.25) rectangle (3,.75) (.75,0) rectangle (1.5,.5);
\draw (0,0) -- (0,.75) -- (2.25,.75);
\node at (2.6,.95) {$d$};\node at (1.9,.9) {$c$};\node at (.4,.9) {$a$};\node at (1.1,.95) {$b$};
\end{tikzpicture}}\nonumber\\
\\
\parbox{6cm}{$F2'$\quad\begin{tikzpicture}
\fill[color=blue!20] (0,1.5) rectangle (.25,3);
\draw (0,0) rectangle (0.25,1.5)  (.25,1.5) rectangle (.5,3) (.25,2.25) rectangle (.75,3) (0,.75) rectangle (.5,1.5);
\draw (0,0) -- (0,3) -- (.75,3);
\node at (.4,.4) {$a$};\node at (.65,1.1) {$b$};\node at (.65,1.85) {$c$};\node at (.9,2.6) {$d$};
\end{tikzpicture}}
\parbox{6cm}{$F2^{\pi'}$\quad\begin{tikzpicture}
\fill[color=blue!20] (0,.75) rectangle (.25,3);
\fill[color=blue!20] (.25,2.25) rectangle (.5,3);
\draw (0,0) rectangle (0.5,.75)  (.25,0) rectangle (.5,2.25) (.5,1.5) rectangle (.75,3) (.25,1.5) rectangle (.75,2.25);
\draw (0,0) -- (0,3) -- (.75,3);
\node at (.65,.4) {$d$};\node at (.65,1.1) {$c$};\node at (.9,1.85) {$b$};\node at (.9,2.6) {$a$};
\end{tikzpicture}.}\nonumber
\end{align}
The diagrams have been arranged so that those on the right hand are the $\pi$--rotations of the diagrams in the left hand, and the diagrams in the second row are the conjugates of the ones in the first row.
\medskip

{\em If these diagrams \eqref{a3} satisfy the additional conditions $a\leq c+1$ and $d\leq b+1$, then they are called $A2, A2^{\pi},  A2'$ and $A2^{\pi'}$ configurations, as illustrated above, replacing the letter $F$ by $A$.}

\medskip

\begin{proposition}\label{prop3}
Let $\lambda/\mu$ be one of the skew diagrams \eqref{a3}. Then, the support of $\lambda/\mu$  coincides with its Schur interval if and only if it is an $A2, A2^{\pi},  A2'$ or a $A2^{\pi'}$ configuration.
\end{proposition}
\begin{proof} By assumption $a, d\ge 1$ and $b,c\ge 0$. When $a=0$ or $d=0$ we are in the case of two columns (rows) were already studied in Proposition \ref{2col}.
Thanks to the rotation and conjugation symmetry, we only consider case $F2'$. We will start by showing that when $a>c+1$ or $d>b+1$, the support of $\lambda/\mu$ is not the entire Schur interval.
For the first case,  $a>c+1$, just note that with $k:=a-(c+1)$, the partition
$$\xi=(d+c+b+k,d+b+c+1)$$
belongs to the Schur interval of $\lambda/\mu$, since the first entry of $\ww$ and $\nn$ is $\min\{d+c+b,b+a\}$ and $d+c+b+a$, respectively.
Moreover, $\xi\notin\sup(\lambda/\mu)$, since when placing in $\lambda/\mu$ the string of length $d+c+b+k$ we must place the integer $d+b+c+1$ in the first column of the diagram, leaving no room to place the second string.

For the case $d>b+1$ note that the entries of $\ww$ are $b+c+d, a+b$ and $d$, and that $\sigma=(a+b+c+d,d,b)$ with $d\geq b+2$. By Lemma \ref{las} it follows that the support of $\lambda/\mu$ is not the entire Schur interval.

Thus, when $a>c+1$ or $d>b+1$ the support of $\lambda/\mu$ does not coincide with its Schur interval.
For the rest of the proof assume $a\leq c+1$ and $d\leq b+1$. Then,
$$\ww=(d+c+b,b+a,d)\preceq\nn=(d+c+b+a,d+b).$$
Let $\xi=(\xi_1,\xi_2,\xi_3)\in[\ww,\nn]$. We will show that $\xi\in\sup(\lambda/\mu)$. Since $\ww\preceq \xi\preceq \nn$, we must have $$d+c+b\leq\xi_1\leq d+c+b+a\quad\text{ and }\quad d+c+2b+a\leq\xi_1+\xi_2\leq 2d+c+2b+a.$$
Then $\xi_1=d+c+b+k$ and  $a+b-k\leq\xi_2\leq d+a+b-k$, for some $k\in\{0,\ldots,a\}$.
 From the Algorithm \ref{Proc:alg},
 this means that, for each $k\in\{0,\ldots,a\}$, after placing (in the unique possible way) the string of length $\xi_1$ in the diagram $\lambda/\mu$, we must insert the strings of length $\xi_2,\xi_3$ in the skew diagram $\tilde \lambda/\mu$, obtained by removing the boxes of the string of length $\xi_1$. The Schur interval of $\tilde \lambda/\mu$ is $[\tilde\ww=(b+a-k,d),\tilde\nn=(d+b+a-k)]$. Since $\tilde \lambda/\mu$ has two columns, from Proposition \ref{2col}, $[\tilde\ww=(b+a-k,d),\tilde\nn=(d+b+a-k)]=\sup(\tilde \lambda/\mu)$.
 On the other hand, $\tilde\ww\preceq(\xi_2,\xi_3)\preceq\tilde\nn$, we have $(\xi_2,\xi_3)\in\sup(\tilde\lambda/\mu)$. Finally, note that when $k\geq 0$,
from the inequality $a-k<c+1$, it follows that $\xi_2\leq d+b+a-k<d+c+b+1$.
Therefore we also have $\xi\in\sup(\lambda/\mu)$.
\end{proof}

\begin{example}
The skew diagram $\lambda/\mu=\tiny\young(::::::\hfill,::\hfill\hfill\hfill\hfill\hfill,\hfill\hfill\hfill\hfill)$ is an $A2^{\pi}$ configuration with
$d=1$ and $c=b=a=2$. Therefore, its Schur interval coincides with its support.
\end{example}


The next configuration that we consider is the ribbon skew diagram $\lambda/\mu$, with $\mu=((a+b+1)^{x},a^y)$ and $\lambda^*=((b+1)^{y+1})$, for some integers $a,b,x,y\geq1$, as illustrated by:
\begin{equation}\label{snake}
F3\quad\begin{tikzpicture}
\fill[color=blue!20] (0,.25) rectangle (.75,1.75);
\fill[color=blue!20] (.75,1) rectangle (1.75,1.75);
\draw (0,0) rectangle (1,.25)  (.75,0) rectangle (1,1) (.75,.75) rectangle (2,1) (1.75,.75) rectangle (2,1.75);
\draw (0,0) -- (0,1.75) -- (2,1.75);
\node at (2.15,1.4) {$x$};\node at (1.15,.47) {$y$};\node at (.4,.4) {$a$};\node at (1.35,1.2) {$b$};
\end{tikzpicture}.
\end{equation}
\medskip

{\em A $F3$ configuration \eqref{snake} with  $a=x=1$, or $a=1$ and $x\leq y+1$, or $a\leq b+1$ and $x=1$,   is called an $A3$ configuration.}
\medskip

\begin{proposition}\label{f3}
Let $\lambda/\mu$ be a skew diagram with a configuration \eqref{snake}. Then its support equals its Schur interval if and only if it is an $A3$ configuration. Moreover, when $a=x=1$,  the support of $A3$ is $[\ww,\nn]=\{\ww,\xi_2=(y+2,2,2,1^{b-1}),\xi_3=(y+2,3,1^b),\xi_1=(y+3,1^{b+2}),\nn\}$, and the skew Schur function $s_{\lambda/\mu}=s_{\ww'}+s_{\xi_1'}+s_{\xi_2'}+s_{\xi_3'}+s_{\nn'}$ has exactly five components all with multiplicity $1$.
\end{proposition}
\begin{proof}
We start by noticing that when both integers $a$ and $x$ are strictly greater than $1$, then the minimum and maximum of the $\sup(\lambda/\mu)$
are given by   $$\ww=(w_1,w_2,1^{a+b})\preceq\nn=(x+y+2,2^{\min\{b+1,a\}},1^{a+b+1-2\min\{b+1,a\}}),$$
with $w_1=\max\{y+2,x+1\}$, $w_2=\min\{y+2,x+1\}$ and $\min\{b+1,a\}\geq 2$. Therefore, we can consider the partition
$$\xi:=(w_1,w_2,3,1^{a+b-3}).$$ It is straightforward to check that $\ww\preceq\xi\preceq\nn$ and that $\xi$ is not in the support of $\lambda/\mu$.

In the case $a=x=1$, in which the minimum and maximum of the support are given by
$$\ww=(y+2,2,1^{b+1})\preceq\nn=(y+3,2,1^b),$$
the Schur interval is  $[\ww,\nn]=\{\ww,(y+3,1^{b+2}),(y+2,2,2,1^{b-1}),(y+2,3,1^b),\nn\}$, and we can check directly that this interval is equal to $\sup(\lambda/\mu)$.

So we are left with the case $a=1$ and $x>1$, since the remaining case is obtained by the conjugation symmetry. The minimum and maximum  of the support are
$$\ww=(w_1,w_2,1^{b+1})\preceq\nn=(x+y+2,2,1^b),$$ where $w_1$ and $w_2$ are defined as in the initial case. Thus, if $\xi$ is a partition  in the Schur interval, it
must satisfy $\xi=(\xi_1,\xi_2,\xi_3,1^b)$, for some integers $\xi_i$. Moreover, we must have
$$\begin{matrix}
\max\{y+2,x+1\}&\le&\xi_1&\le&x+y+2,\\
x+y+3&\le&\xi_1+\xi_2&\le&x+y+4,\\
&&\text{and}&&\\
x+y+4&\le&\xi_1+\xi_2+\xi_3&\le &x+y+5.
\end{matrix}$$
There are two possibilities for the sum $\xi_1+\xi_2$. As $\xi_3\ge 1$, when this sum equals $x+y+3$, it follows that $\xi_3$ must be either 1 or 2, and when $\xi_1+\xi_2=x+y+4$ then $\xi_3=1$. Consider $\xi_1+\xi_2=x+y+3$.  We have $\tilde\ww_1:=(w_1,w_2)\preceq (\xi_1,\xi_2)\preceq \tilde\nn_1:=(x+y+2,1)$, with $\tilde\ww_1$ and $\tilde\nn_1$, respectively, the minimum and maximum LR fillings of the two column skew diagram $\tilde A$ obtained from $\lambda/\mu$ by removing all except the second and last columns. Since, by Proposition \ref{2col}, we have $\sup(\tilde A)=[\tilde\ww_1,\tilde\nn_1]$, we conclude that we may place  strings of length $\xi_1$ and $\xi_2$ using only the second and last columns of $\lambda/\mu$. As the remaining entries of $\xi$ are equal to $1$, and,  at most one, equal to 2,  it follows that $\xi\in\sup(\lambda/\mu)$.

Assume now that $\xi_1+\xi_2=x+y+4$. In this case, $(\xi_1,\xi_2)\in[\tilde\ww_2,\tilde\nn_2]$, where $\tilde\ww_2=(w_1,w_2,1)$ and $\tilde\nn_2=(x+y+2,2)$ are the minimum and maximum LR fillings of the skew diagram $B$ obtained from $\lambda/\mu$ removing all columns except  the first two and the last one.  Note that $B$ is an $F2^{\pi'}$ \eqref{a3}, and, by Proposition  \ref{prop3}, $\sup(B)=[\tilde\ww_2,\tilde\nn_2]$ if and only if $x\leq y+1$.
Thus, when $x\leq y+1$ we find that $(\xi_1,\xi_2)\in\sup(B)$ and similarly,  as before, it follows that $\xi\in\sup(\lambda/\mu)$. If, on the other hand, we have $x>y+1$, then, by Proposition \ref{prop3}, we can consider a partition $(\sigma_1,\sigma_2)\in[\tilde\ww_2,\tilde\nn_2]$ which is not in the set $\sup(B)$. It follows that $\sigma:=(\sigma_1,\sigma_2,1^b)\in[\ww,\nn]$ but $\sigma\notin\sup(\lambda/\mu)$.
\end{proof}


In the next lemmas we analyse some families of skew diagrams  needed in the sequel.
We start with skew diagrams $\lambda/\mu$ of types $F4$ and $\widetilde{F}4$ defined respectively by the partitions
$\lambda=((a+2)^x,a+1,1^y)$  and $\mu=((a+1)^x)$ for some $a,\,x,\,y\geq 1$ such that not both $x$ and $y$ are equal to 1,
and  by the partitions
$\lambda=((a+b+1)^x,a+b,a) $ and $\mu=((a+b)^x)$, for some integers $b\geq 1$ and $a,x>1$,  as illustrated by:

\begin{equation}\label{pc1}
\parbox{5cm}{$F4$\quad\begin{tikzpicture}
\fill[color=blue!20] (0,1) rectangle (1,1.75);
\draw (0,0) rectangle (.25,1)  (0,.75) rectangle (1,1) (1,1) rectangle (1.25,1.75);
\draw (0,0) -- (0,1.75) -- (1.25,1.75);
\node at (1.4,1.4) {$x$};\node at (.4,.4) {$y$};\node at (.6,1.15) {$a$};
\end{tikzpicture}}
\parbox{5cm}{$\widetilde{F}4$\quad\begin{tikzpicture}
\fill[color=blue!20] (0,.5) rectangle (2,1.75);
\draw (0,0) rectangle (1,.5)  (0,.25) rectangle (2,.5) (2,.5) rectangle (2.25,1.75);
\draw (0,0) -- (0,1.75) -- (2.25,1.75);
\node at (2.4,1.1) {$x$};\node at (.5,.65) {$a$};\node at (1.5,.7) {$b$};
\end{tikzpicture}.}
\end{equation}
Note that if we let $x=y=1$ in an $F4$ configuration, or $x=1$ in an $\widetilde{F}4$ configuration, then we get an $F2$ configuration.
 \medskip

 {\em An $F4$ configuration with $a=1$ and $x\leq y+1$, or $a\geq 2$ and $x=1$, is called an $A4$ configuration.}

\medskip

\begin{proposition}\label{f4}
$(i)$ If $\lambda/\mu$ is a skew diagram with configuration $F4$, then its support is equal to  the Schur interval if and only if it is an $A4$ configuration. Moreover, when $a\geq 2$ and $x=1$, the support of $A4$ is $[\ww,\nn]=\{\ww,\xi=(y+1,2,1^{a-1}),\nn\}$ and the skew Schur function
$s_{\lambda/\mu}=s_{\ww'}+s_{\xi'}+s_{\nn'}$ has exactly three components all with multiplicity $1$.

$(ii)$ The support of  $\widetilde{F}4$  is strictly contained in the Schur interval.
\end{proposition}
\begin{proof}
We start with an $F4$ configuration. If $a\geq 2$ then the minimum and maximum  of $\sup(\lambda/\mu)$ are
$$\ww=(w_1,w_2,1^a)\preceq\nn=(x+y+1,1^a),$$
with $w_1=\max\{x,y+1\}$, $w_2=\min\{x,y+1\}$ and $\ell(\ww)=\ell(\nn)+1$. When $x\geq 2$ the partition $\xi:=(w_1,w_2,2,1^{a-2})$ satisfy $\ell(\xi)=\ell(\nn)$, and thus $\ww\preceq\xi\preceq\nn$, but is not in  the support of $\lambda/\mu$, since the strings of length $w_1$ and $w_2$ must fill the first and last columns, leaving no space for the string of length 2. On the other hand, if $x=1$, then the Schur interval of $\lambda/\mu$ is $$[\ww,\nn]=\{\ww,(y+1,2,1^{a-1}),\nn\}=\sup(\lambda/\mu).$$

Assume now that $a=1$. If $x>y+1$ then $\ww=(x,y+1,1)\preceq\nn=(x+y+1,1),$
and $\xi:=(x,y+2)\in[\ww,\nn]$ but clearly $\xi\notin\sup(\lambda/\mu)$. If otherwise, we have $x\leq y+1$, then
the minimum and the maximum  of $\lambda/\mu$ are given by
$$\ww=(y+1,x,1)\preceq\nn=(x+y+1,1).$$
A partition $\xi=(\xi_1,\xi_2,\xi_3)\in[\ww,\nn]$ must satisfy $y+1\leq \xi_1\leq x+y+1$ and $x+y+1\leq \xi_1+\xi_2\leq x+y+2$ with $\xi_3\in\{0,1\}$. Let $\xi_1=y+1+k$, for some $k\in\{0,\ldots,y+1\}$. Then, we get $x-k\leq\xi_2\leq x+1-k$, and since $x\leq y+1$ it follows that the partition $\xi$ belongs to  the support of $\lambda/\mu$.

\medskip

Finally, suppose next that $\lambda/\mu$ is an $\widetilde{F}4$ configuration. Then the minimum and maximum of $\sup(\lambda/\mu)$ are
$$\ww=(x,2^a,1^b)\preceq\nn=(x+2,2^{a-1},1^b),$$
and, since $a\ge 2$, $\xi:=(x+2,2^{a-2},1^{b+2})$ is a partition and satisfy $\ww\preceq\xi\preceq\nn$, but $\xi\notin\sup(\lambda/\mu)$.
\end{proof}

The next skew diagrams $\lambda/\mu$ are respectively  $F5$, $\widetilde{F}5$ and $\widehat{F}5$, defined by the partitions: $\mu=((a+b)^{x+y})$ and $\lambda^*=(b+2,1^{y+1})$, with $a,x\geq 2$ and  $b,y\geq 0$;  $\mu=((a+1)^{x+y})$ and $\lambda^*=((a+2)^z,1^{y+1})$, with $a\geq 2, y\geq 0$ and  $x,z\geq 1$; and
$\mu=((a+b+c)^x,a)$ and $\lambda^*=(c+1)$, with $x\geq 2$ and either  $a,b,c\geq 1$ or $b=0$ and $a,c\geq 1$ with $a+c\geq 3$, as illustrated by:

\begin{equation}\label{pc2}
\parbox{5cm}{$F5$\quad\begin{tikzpicture}
\fill[color=blue!20] (0,.5) rectangle (2,2);
\draw (0,0) rectangle (1,.5)  (0,.25) rectangle (2.25,.5) (2,.25) rectangle (2.25,2) (2,1) rectangle (2.5,2);
\draw (0,0) -- (0,2) -- (2.5,2);
\node at (2.4,.75) {$y$};\node at (2.65,1.5) {$x$};\node at (.5,.65) {$a$};\node at (1.5,.7) {$b$};
\end{tikzpicture}}
\parbox{4cm}{$\widetilde{F}5$\quad\begin{tikzpicture}
\fill[color=blue!20] (1,.5) rectangle (2,2);
\draw (1,-.5) rectangle (1.25,.5)  (1,.25) rectangle (2.25,.5) (2,.25) rectangle (2.25,2) (2,1) rectangle (2.5,2);
\draw (1,0) -- (1,2) -- (2.5,2);
\node at (1.4,-.15) {$z$};\node at (2.4,.75) {$y$};\node at (2.65,1.5) {$x$};\node at (1.65,.65) {$a$};
\end{tikzpicture}}
\parbox{4cm}{$\widehat{F}5$\quad\begin{tikzpicture}
\fill[color=blue!20] (0,.25) rectangle (.6,1.25);\fill[color=blue!20] (.6,.5) rectangle (1.75,1.25);
\draw (0,0) rectangle (1,.25)  (.6,0) rectangle (1,.5) (1,.25) rectangle (2,.5) (1.75,.25) rectangle (2,1.25);
\draw (0,0) -- (0,1.25) -- (2,1.25);
\node at (2.15,.9) {$x$};\node at (.35,.4) {$a$};\node at (.8,.7) {$b$};\node at (1.4,.65) {$c$};
\end{tikzpicture}.}
\end{equation}

\begin{lemma}\label{f5}
The support of  $F5, \,\widetilde{F}5$ or  $\widehat{F}5$  is strictly contained in the Schur interval.
\end{lemma}
\begin{proof}
The minimum and maximum  of the support of $F5$ are respectively
$$\ww=(x+y+1,x,2^a,1^b)\preceq\nn=(x+y+2,x+2,2^{a-2},1^{b+1}),$$
with $\ell(\ww)=\ell(\nn)+1$. Since $a,\,x>1$, we may consider the partition $$\xi:=(x+y+1,x+1,3,2^{a-2},1^b)$$ which is an element of the Schur interval but it is not in  the support of $F5$.

In the case of $\widetilde{F}5$, we have
$$\ww=(w_1,w_2,w_3,1^a)\preceq\nn=(x+y+z+1,x+1,1^a),$$
\noindent with $w_1,w_2,w_3$ the lengths of the first and the two last columns of $\widetilde{F}5$ by decreasing order. Take $$\xi:=(x+y+z+1,x,2,1^{a-1})\in[\ww,\nn].$$
Since after placing the strings of length $x+y+z+1$ and $x$, in the unique possible positions in $\widetilde{F}5$, we are left with only a single row, it follows that $\xi\notin\sup(\widetilde{F}5)$.

Finally,  we are in the case of $\widehat{F}5$. If $b\geq 1$, the minimum and the maximum of the support are given by
$$\ww=(x+1,2^b,1^{a+c})\preceq\nn=(x+2,2^b,2^{\min\{a-1,c\}},1^{a-1+c-2\min\{a-1,c\}}).$$
It is easy to check that the partition $\xi:=(x+1,3,2^{b-1}1^{a+c-1})$ is an element of the Schur interval but is not in  the support of $\widehat{F}5$. Similarly, when $b=0$ the partition $\xi=(x+1,3,1^{a+c-3})$ belongs to the Schur interval but not to the support of $\widehat{F}5$.
\end{proof}

The next family of skew diagrams $\lambda/\mu$ is designated by $F6$ and it is defined by partitions $\lambda=(a+b+1,(a+1)^x,1^y)$ and $\mu=(1)$, for some
integers $a,x>0$ and $b,y\geq 1$,
\medskip

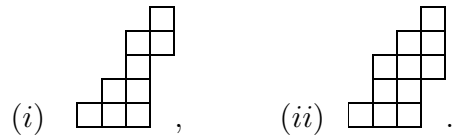
\begin{equation}\label{f6}
F6\quad\begin{tikzpicture}
\fill[color=blue!20] (0,1.2) rectangle (.25,1.45);
\draw (0,1.2) -- (0,1.45) -- (.25,1.45);
\draw (0,0) rectangle (.25,1.2) (.25,1.2) rectangle (1.4,1.45) (0,.8) rectangle (1,1.45);
\node at (-0.15,0.4) {$y$};\node at (-0.15,1) {$x$};\node at (.6,1.6) {$a$};\node at (1.15,1.65) {$b$};
\end{tikzpicture}.
\end{equation}

{\em  When $b=y=1$ an $F6$ configuration  is called an $A6$ configuration.}

\begin{proposition}
The support of a skew diagram $\lambda/\mu$ with an $F6$ configuration equals the Schur interval if and only if it is an $A6$ configuration.
Moreover, in this case, the support is $[\ww,\nn]=\{\ww=((x+1)^{a+1},1),\xi=(x+2,(x+1)^{a-1},x,1),\nn=(x+2,(x+1)^a):a,\,x\ge 1\}$, and the skew Schur function $s_{\lambda/\mu}=s_{\ww'}+s_{\xi'}+s_{\nn'}$ has three components all with multiplicity $1$.
\end{proposition}
\begin{proof}
We consider the case $b\geq 2$. Notice that the case $y\geq 2$ is the conjugate of the former. In our case,
 the minimum and the maximum of the support are  respectively $$\ww=(x+y,(x+1)^a,1^b)\preceq\nn=(x+y+1,(x+1)^a,1^{b-1}),$$
\noindent and we can consider the partition $\xi:=(x+y,(x+1)^a,2,1^{b-2})\in[\ww,\nn]$, which clearly does not belong to the support of $\lambda/\mu$.

For the remaining  case $y=b=1$, notice that the Schur interval is given by
$$[\ww,\nn]=\{\ww=((x+1)^{a+1},1),\xi=(x+2,(x+1)^{a-1},x,1),\nn=(x+2,(x+1)^a)\}.$$
Since $\xi\in\sup(\lambda/\mu)$ the result follows.
\end{proof}



In the next lemma we analyse the support of a skew diagram $\lambda/\mu$, with $\ell+1\geq 4$ columns, having the form
\begin{equation}\label{cf8}
F7\quad\begin{tikzpicture}
\fill[color=blue!20] (0,.4) rectangle (.25,2.1);
\fill[color=blue!20] (.25,1) rectangle (.5,2.1);
\fill[color=blue!20] (.5,1.7) rectangle (.75,2.1);
\draw (0,0) -- (0,.4) -- (.25,.4) -- (.25,1) -- (.5,1) -- (.5,1.7) --
(.75,1.7) -- (.75,0) -- (0,0);
\draw (.75,.5) rectangle (1,2.1);
\draw (0,0) -- (0,2.1) -- (1,2.1);
\node at (0.9,0.2) {$y$};\node at (1.15,1.2) {$x$};
\end{tikzpicture},
\end{equation}
where the first $\ell$ columns end in the same row and have pairwise
distinct lengths, $x$ is the length of the last column and
$\lambda^*=(1^y)$. Moreover by Lemma \ref{bloco}, we assume without loss of
generality that the last but one column starts at least one row below the
topmost box of the last column, and that the first column has length $\leq
y$. Denote by $(w_1,\ldots,w_{\ell})$ the partition formed by the first
$\ell$ columns of the diagram, and let $k\geq 0$ be the number of rows
that the last two columns share.
\medskip

{\em A skew diagram $F7$ such that $\ell=3,\,
x=w_2,\, w_3=1$ and $k=w_2-1$ is called an $A7$ configuration.}
\vskip 0.3cm
 We
distinguish two cases: either $(i)$ $w_2\leq y$, or $(ii)$ $w_2>y$.
In the
first case, the last column shares rows only with column $\ell$, and, in
the second case, the last column shares rows with at least columns $\ell$
and $\ell-1$. Examples of $A7$ configurations of types $(i)$ and $(ii)$
are shown below:

 $$(i)\quad\tiny\young(:::\hfill,::\hfill\hfill,::\hfill,:\hfill\hfill,\hfill\hfill\hfill),\qquad\qquad
(ii)\quad
\tiny\young(:::\hfill,::\hfill\hfill,:\hfill\hfill\hfill,:\hfill\hfill,\hfill\hfill\hfill).$$

In the next two lemmas we show that the support of an $A7$ configuration
is the entire Schur interval.

\begin{lemma}\label{f81}
Let $\lambda/\mu$ be an $F7$ configuration \eqref{cf8} such that $w_2\leq
y$. Then, the support of $\lambda/\mu$ equals the Schur interval if and
only if $\ell=3,\, w_3=1,\, x=w_2$ and $k=w_2-1$.
\end{lemma}
\begin{proof}
The condition $w_2\leq y$ means that the last column of $\lambda/\mu$
shares rows with at most the last but one column.  Then,
the minimum and maximum of the support of $\lambda/\mu$  are given respectively by
$$\ww=(w_1,\ldots,w_{\ell})\cup\{x\}\preceq\nn=(w_1+x-k,w_2+k,w_3,\ldots,w_{\ell}),$$
with $\ell(\ww)=\ell(\nn)+1$.

Note that when $w_{\ell}\geq 2$, it follows, from Corollary \ref{same}, that
$\sup(\lambda/\mu)\varsubsetneq[\ww,\nn]$, and when $\ell\geq 4$ and
$w_{\ell}=1$, the partition
$\xi:=(w_1+x-k,w_2+k,w_3,\ldots,w_{\ell-1}-1,1,1)$ shows that
$\sup(\lambda/\mu)\varsubsetneq[\ww,\nn]$. Moreover, if $k=0$ the skew
diagram is disconnected, and, by Lemma \ref{disconnected}, the support of
$\lambda/\mu$ is not the entire Schur interval.

So, assuming $\ell=3,\, k>0$ and $w_3=1$, we have
$$\ww=(w_1,w_2,1)\cup\{x\}\preceq\nn=(w_1+x-k,w_2+k,1).$$

If $x<w_2$ then $1\leq k<x<w_2$, and, in particular, we get $w_2\geq k+2$.
Since $w_2$ and $k$ are the lengths of the second and third columns of
 $\lambda^1/\mu$ (see page \pageref{deflambda1} for the definition of $\lambda^1/\mu$), it follows, from Lemma \ref{las}, that
$\sup(\lambda/\mu)\varsubsetneq[\ww,\nn]$. The same situation happens if
$x>w_2$, since in this case the partition $\xi:=(w_1,w_2+1)\cup(x)$
satisfies $\ww\preceq\xi\preceq\nn$, but it does not belong to the support
of $\lambda/\mu$.

We are, therefore, left with the case $\ell=3, k>0, x=w_2$ and $w_3=1$.
Note that $1\leq k\leq w_2-1$. When $k<w_2-1$, the second and third
columns of $\lambda^1/\mu$ have lengths $w_2$ and $k$, respectively, and by
Lemma \ref{las}, we find that the support of $\lambda/\mu$ is not the
entire Schur interval.
So we must also consider $k=w_2-1$. In this case, the minimum and the maximum
 of the support are given by
$$\ww=(w_1,w_2,w_2,1)\preceq\nn=(w_1+1,w_2+w_2-1,1).$$ Let
$\xi:=(\xi_1,\xi_2,\xi_3,\xi_4)\in[\ww,\nn]$, and note that
$w_1\leq\xi_1\leq w_1+1$.

If $\xi_1=w_1$, then from the inequalities $\ww\preceq\xi\preceq\nn$ it
follows that $\alpha\preceq(\xi_2,\xi_3,\xi_4)\preceq\beta$, where
$\alpha$ and $\beta$ are the minimum and the maximum  of the support of the skew
diagram $A$ obtained from $\lambda/\mu$ by removing the third column. Since
$A$ is an $A2^{\pi'}$ configuration, it follows that
$(\xi_2,\xi_3,\xi_4)\in\sup(A)$, and therefore $\xi\in\sup(\lambda/\mu)$.

For the remaining case $\xi_1=w_1+1$ the situation is analogous, since in
this case we have $\ww^1\preceq(\xi_2,\xi_3,\xi_4)\preceq\nn^1$,  where
$\ww^1$ and $\nn^1$ give the minimum and the maximum LR filling of
$A\setminus V_1$. Since this diagram is also an $A2^{\pi'}$ configuration,
we find that $\xi\in\sup(\lambda/\mu)$.
\end{proof}

\begin{lemma}\label{f82}
Let $\lambda/\mu$ be an $F7$ configuration \eqref{cf8} such that $w_2>y$.
Then, the support of $\lambda/\mu$ equals the Schur interval if and only
if $\ell=3,\, w_3=1, \,x=w_2$ and $k=w_2-1$.
\end{lemma}
\begin{proof}
We start the proof by showing that if the conditions $\ell=3,\, w_3=1,\,
x=w_2$ and $k=w_2-1$ are not satisfied then the support of $\lambda/\mu$
is not the entire Schur interval. Consider the minimal and maximal
fillings of the diagram
$$\ww=(w_1,w_2,\ldots,w_{\ell})\cup(x)\preceq\nn=(n_1,n_2,\ldots,n_s),$$
where $n_1=x+y,\, n_2=w_1$ and $\ell(\ww)>\ell(\nn)$.

Note that when $k=0$, the diagram $\lambda/\mu$ is disconnected, with one
of the connected components having a 2 by 2 block. In this case, by Lemma
\ref{disconnected}, $\sup(\lambda/\mu)$ is not the entire Schur
interval. So, assume $k>0$.
If $x>w_2$ then the partition
$\xi:=(w_1,\ldots,w_{\ell-2},w_{\ell-1}+1,w_{\ell}-1)\cup(x)$ shows that
$\sup(\lambda/\mu)$ is strictly contained in the Schur interval
$[\ww,\nn]$. If, on the other hand, $x<w_2$ then $1\leq k<x<w_2$, and this
implies $w_2\geq k+2$. Since  $k$ and $w_2$ are the lengths of two columns
of $\lambda^1/\mu$, it follows, from Lemma \ref{las}, that also in this case
the support of $\lambda/\mu$ is not the entire Schur interval.

So, for the rest of the proof we assume $x=w_2$, and therefore
$$\ww=(w_1,w_2,w_2,w_3,\ldots,w_{\ell})\preceq\nn=(n_1,n_2,\ldots,n_s),$$
with $n_1=w_2+y$ and $n_2=w_1$. Note that this implies $y\geq 2$.

If $\ell\geq 4$ then the partition
$\xi:=(w_1,w_2,w_2,\ldots,w_{\ell-2},w_{\ell-1}+1,w_{\ell}-1)$ clearly
shows that $\sup(\lambda/\mu)$ is not the entire Schur interval.
Thus, consider $\ell=3$ and note that since $1\leq k\leq w_2-1$, it
follows, from Lemma \ref{las}, that $\sup(\lambda/\mu)\subsetneq[\ww,\nn]$
except if $k=w_2-1$.

So, assume now that $\ell=3, x=w_2$ and $k=w_2-1$. Then,
$n_1=w_2+y=w_1+1$, and the minimal and maximal fillings are now
$$\ww=(w_1,w_2,w_2,w_3)\preceq\nn=(w_1+1,w_1,w_3+h),$$ where $h>0$ is the
number of rows that the second and the last columns share.
It follows that if $w_3\geq 2$ the partition
$\xi:=(w_1,w_2+1,w_2+1,w_3-2)$ satisfies $\ww\preceq\xi\preceq\nn$ but it is
not in the support of $\lambda/\mu$.

To finish the proof consider  $\ell=3, w_3=1, x=w_2$ and $k=w_2-1$, and let
 $\xi=(\xi_1,\xi_2,\xi_3,\xi_4)$ be a partition in the Schur interval
$[\ww,\nn]$. Using the same argument  used in the proof of the previous
Lemma \ref{f81}, it is easy to show that $\xi$ belongs to the support of
$\lambda/\mu$, and it follows that in this case the support of
$\lambda/\mu$ is the entire Schur interval.
\end{proof}

From  lemmas \ref{f81} and \ref{f82} we deduce the following result.

\begin{corollary}\label{f8}
If the skew diagram $\lambda/\mu$ is an $F7$ configuration, then its support is the Schur interval if and only if $\lambda/\mu$ is an $A7$ configuration.
\end{corollary}


\section{Full interval linear expansion of multiplicity--free skew Schur functions. }

\medskip

We are now ready to identify the basic multiplicity--free skew Schur functions whose support is the entire interval $[{\bf w},{\bf n}]$. Our strategy and terminology follows closely the one used in the proof of Lemma $7.1$ in \cite{DTK}.

\noindent {\bf Proof of Theorem \ref{main}}.
    If $\lambda/\mu$ satisfy one or more of the conditions listed in the theorem, then the corresponding Schur function is multiplicity--free.
    They are particular instances of the configurations  ${R0-R4}$ described in Theorem \ref{schurfree} as follows : $A1$ is in $R4$; $A2$ is in $R1$, or $R3$; $A3$ and $A4$ are in $R3$; $A6$ is in $R1$ or $R3$; and $A7$ is in $R1$.
    The strategy for the reciprocal is to use  Corollary \ref{TF3}  to analyse the support of the basic multiplicity--free skew Schur function $s_{\lambda/\mu}$ listed under cases ${ R0-R4}$ in Theorem \ref{schurfree}.  We next consider, in bold, the five cases, ${ \bf R0-R4}$, in turn.

\medskip

 ${\bf R0}$. If $\mu$ or $\lambda^*$ is the zero partition $0$ then either $\lambda/\mu$ or $(\lambda/\mu)^{\pi}$ is a partition. This means that the minimum and maximum LR fillings of $\lambda/\mu$ coincide and therefore, $\sup(\lambda/\mu)=[{\bf w},{\bf n}]=\{\nn=\ww\}$.

\bigskip

 ${\bf R4}$. In this case both $\mu$ and $\lambda^*$ are rectangles, and thus $\lambda$ is a fat hook. Thanks to Lemma \ref{bloco}, we may assume that $\lambda=((a+b)^x,b^y)$ and $\mu=(a^x,0^y)$ with $a,b,x,y\geq 1$, as illustrated below:
  $$\lambda/\mu=
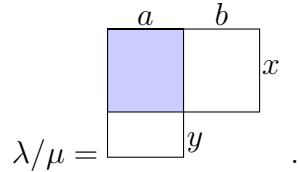
\begin{tikzpicture}
\fill[color=blue!20] (0,.6) rectangle (1,1.7);
\draw (0,0) rectangle (1,1.7)  (0,.6) rectangle (2,1.7);
\node at (1.15,0.2) {$y$};\node at (2.15,1.2) {$x$};\node at (.5,1.85) {$a$};\node at (1.5,1.9) {$b$};
\end{tikzpicture}.$$

Since $\lambda/\mu$ has two disconnected components, by Lemma \ref{disconnected}, it follows that if any of the components has a 2 by 2 block then its support is not the entire Schur interval. Thus, we are left with four cases to analyse. When $a=b=1$ or when $x=y=1$ we get a two column or a two row diagram. In both cases, by Proposition \ref{2col}, the Schur interval coincides with the support of $\lambda/\mu$.  In any other case we get either an $F1$ or an $F1^{\pi}$ configuration, and, by Proposition \ref{prop1}, we find that the support of  $\lambda/\mu$ is strictly contained in its Schur interval.

Therefore, by Lemma \ref{bloco}, we find that if both $\mu$ and $\lambda^*$ are rectangles, then $\sup(\lambda/\mu)=[\ww,\nn]=\{\nn,\ww\}$ if and only if  $\lambda/\mu$ satisfy conditions $(ii)$ of the theorem.

\bigskip

 ${\bf R2}$. The two main subcases, $\mu$ a rectangle of $m^n$--shortness 2 and $\lambda^*$ a fat hook, and
vice versa, are denoted by ${\bf S2}$ and ${\bf S2^{\pi}}$, respectively. Each  has four subcases, as
illustrated by:

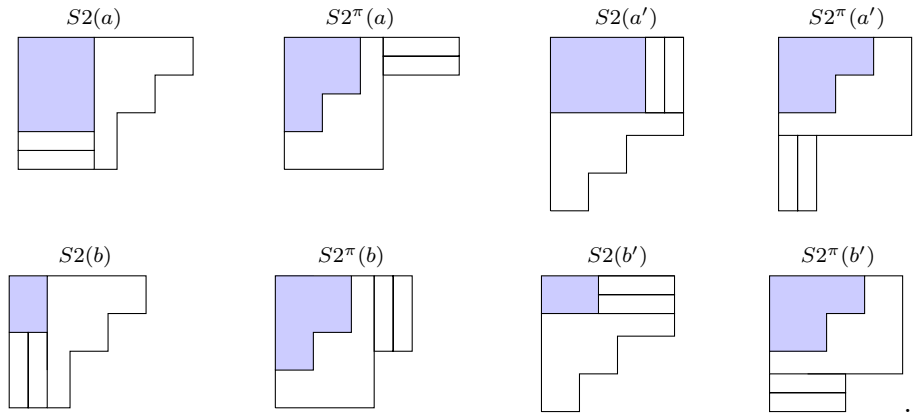
\begin{align*}
\begin{tikzpicture}
\fill[color=blue!20] (0,.5) rectangle (1,1.75);
\draw (0,0) rectangle (1,.25)  (0,.25) rectangle (1,.5);
\draw (1,0) -- (1.3,0) -- (1.3,.75) -- (1.8,.75) -- (1.8,1.25) -- (2.3,1.25) -- (2.3,1.75) -- (1,1.75) -- (1,.5);
\draw (0,.5) -- (0,1.75) -- (1,1.75);
\node at (1,2) {${\scriptstyle S2(a)}$};
\fill[color=blue!20] (3.5,.5) rectangle (4,1.75);
\fill[color=blue!20] (4,1) rectangle (4.5,1.75);
\draw (4.8,1.25) rectangle (5.8,1.5)  (4.8,1.5) rectangle (5.8,1.75);
\draw (3.5,0) -- (4.8,0) -- (4.8,1.75) -- (4.5,1.75) -- (4.5,1) -- (4,1) -- (4,.5) -- (3.5,.5) -- (3.5,0);
\draw (3.5,.5) -- (3.5,1.75) -- (4.5,1.75);
\node at (4.5,2) {${\scriptstyle S2^{\pi}(a)}$};
\fill[color=blue!20] (7,.75) rectangle (8.25,1.75);
\draw (8.25,.75) rectangle (8.5,1.75)  (8.5,.75) rectangle (8.75,1.75);
\draw (7,-.55) -- (7.5,-.55) -- (7.5,-.05) -- (8,-.05) -- (8,.45) -- (8.75,.45) -- (8.75,.75) -- (7,.75) -- (7,-.55);
\draw (7,.75) -- (7,1.75) -- (8.75,1.75);
\node at (8,2) {${\scriptstyle S2(a')}$};
\fill[color=blue!20] (10,.75) rectangle (10.75,1.75);
\fill[color=blue!20] (10.75,1.25) rectangle (11.25,1.75);
\draw (10,-.55) rectangle (10.25,.45)  (10.25,-.55) rectangle (10.5,.45);
\draw (10.5,.45) -- (11.75,.45) -- (11.75,1.75) -- (11.25,1.75) -- (11.25,1.25) -- (10.75,1.25) -- (10.75,.75) -- (10,.75) -- (10,.45);
\draw (10,.75) -- (10,1.75) -- (11.25,1.75);
\node at (10.9,2) {${\scriptstyle S2^{\pi}(a')}$};
\end{tikzpicture}\\
\begin{tikzpicture}
\fill[color=blue!20] (0,1) rectangle (.5,1.75);
\draw (.25,0) rectangle (.5,1)  (0,0) rectangle (.25,1);
\draw (.5,0) -- (.8,0) -- (.8,.75) -- (1.3,.75) -- (1.3,1.25) -- (1.8,1.25) -- (1.8,1.75) -- (.5,1.75) -- (.5,.5);
\draw (0,1) -- (0,1.75) -- (.5,1.75);
\node at (1,2) {${\scriptstyle S2(b)}$};
\fill[color=blue!20] (3.5,.5) rectangle (4,1.75);
\fill[color=blue!20] (4,1) rectangle (4.5,1.75);
\draw (4.8,1.75) rectangle (5.05,.75)  (5.05,1.75) rectangle (5.3,.75);
\draw (3.5,0) -- (4.8,0) -- (4.8,1.75) -- (4.5,1.75) -- (4.5,1) -- (4,1) -- (4,.5) -- (3.5,.5) -- (3.5,0);
\draw (3.5,.5) -- (3.5,1.75) -- (4.5,1.75);
\node at (4.5,2) {${\scriptstyle S2^{\pi}(b)}$};
\fill[color=blue!20] (7,1.25) rectangle (7.75,1.75);
\draw (7.75,1.25) rectangle (8.75,1.5)  (7.75,1.5) rectangle (8.75,1.75);
\draw (7,-.05) -- (7.5,-.05) -- (7.5,.45) -- (8,.45) -- (8,.95) -- (8.75,.95) -- (8.75,1.25) -- (7,1.25) -- (7,-.05);
\draw (7,1.25) -- (7,1.75) -- (8.75,1.75);
\node at (8,2) {${\scriptstyle S2(b')}$};
\fill[color=blue!20] (10,.75) rectangle (10.75,1.75);
\fill[color=blue!20] (10.75,1.25) rectangle (11.25,1.75);
\draw (10,.45) rectangle (11,.2)  (10,.2) rectangle (11,-.05);
\draw (10.5,.45) -- (11.75,.45) -- (11.75,1.75) -- (11.25,1.75) -- (11.25,1.25) -- (10.75,1.25) -- (10.75,.75) -- (10,.75) -- (10,.45);
\draw (10,.75) -- (10,1.75) -- (11.25,1.75);
\node at (10.9,2) {${\scriptstyle S2^{\pi}(b')}$};
\end{tikzpicture}.
\end{align*}

These skew Young diagrams are arranged so that those of type $S2^{\pi}$
are  the $\pi$--rotations of those of type $S2$, and the right--hand block of four is  the
conjugate of the left--hand block of four. Thanks to the rotation symmetry and the
conjugate symmetry, one has only  to consider two cases, which we
select to be ${\bf S2(a)}$ and ${\bf S2(b)}$.

\medskip

${\bf S2(a)}$. In this case $\lambda=((a+b+c+d)^x,(a+b+c)^y,(a+b)^z)$ and $\mu=(a^{x+y+z-2})$ with $x+y+z\geq 4$, for some integers such that $a\geq 2$, $c,d,x,y,z\geq 1$ and $b\geq 0$, as illustrated in the following figure:

$$
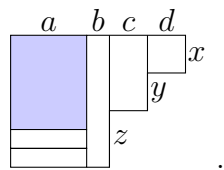
\begin{tikzpicture}
\fill[color=blue!20] (0,.5) rectangle (1,1.75);
\draw (0,0) rectangle (1,.25)  (0,.25) rectangle (1,.5);
\draw (1,0) -- (1.3,0) -- (1.3,.75) -- (1.8,.75) -- (1.8,1.25) -- (2.3,1.25) -- (2.3,1.75) -- (1,1.75) -- (1,.5);
\draw (0,.5) -- (0,1.75) -- (1,1.75);
\draw (1.3,0) -- (1.3,1.75);
\draw (1.8,.75) -- (1.8,1.75);
\node at (.5,1.9) {$a$}; \node at (1.15,1.95) {$b$}; \node at (1.55,1.9) {$c$}; \node at (2.05,1.95) {$d$};
\node at (2.45,1.5) {$x$}; \node at (1.95,1) {$y$}; \node at (1.45,.4) {$z$};
\end{tikzpicture}.$$
By Lemma \ref{bloco}, we may assume that $b=0$.
We start by noticing that when $z=2$ the skew diagram $\lambda/\mu$ is disconnected with a 2 by 2 block,  in  which case we have $\sup(\lambda/\mu)\varsubsetneq[\ww,\nn]$, by Lemma \ref{disconnected}. Assume now that $z=1$, and note that by the hypothesis on $\lambda/\mu$, the length of the diagram is greater than, or equal to 4. Therefore, an $F0^{\pi}$ or an $F0'$ configuration appears when $c\geq 2$ or when $x,\,d\geq 2$, respectively. Again in this cases the support of $\lambda/\mu$ is strictly contained in the Schur interval by Corollary \ref{TF3}.  Now, when $c=x=1$, we get an $F5$ configuration, and when $c=d=1$, we get the conjugate of an ${F}5$ configuration. By Lemma \ref{f5}, $\sup(\lambda/\mu)\varsubsetneq[\ww,\nn]$. Therefore, in all subcases, the support of the skew diagram ${\bf S2(a)}$ is strictly contained in its Schur interval.

The analysis of case ${\bf S2(b)}$ is completely analogous to the previous one.


\bigskip

 ${\bf R3.}$ The two main subcases, $\mu$  a rectangle and $\lambda^*$ a fat hook of $m^n$--shortness 1, and
vice versa,  are denoted by ${\bf S3}$ and ${\bf S3^{\pi}}$, respectively. Each has six subcases, as
illustrated by:

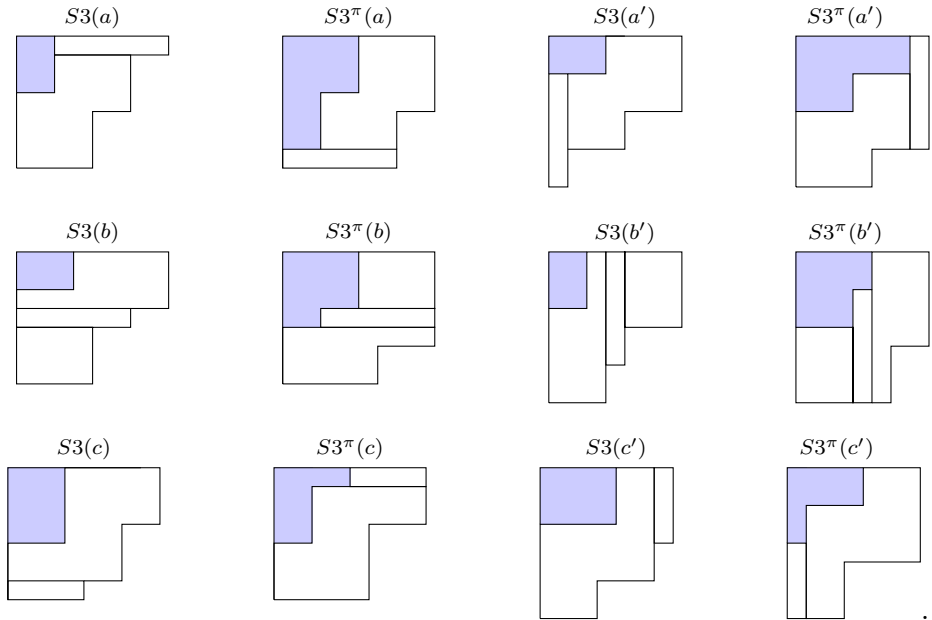
\begin{align*}
\begin{tikzpicture}
\fill[color=blue!20] (0,1) rectangle (.5,1.75);
\draw (.5,1.5) rectangle (2,1.75);
\draw (0,0) -- (1,0) -- (1,.75) -- (1.5,.75) -- (1.5,1.5) -- (.5,1.5) -- (.5,1) -- (0,1) -- (0,0);
\draw (0,1) -- (0,1.75) -- (.5,1.75);
\node at (1,2) {${\scriptstyle S3(a)}$};
\fill[color=blue!20] (3.5,.25) rectangle (4,1.75);
\fill[color=blue!20] (4,1) rectangle (4.5,1.75);
\draw (3.5,0) rectangle (5,.25);
\draw (4,.25) -- (4,1) -- (4.5,1) -- (4.5,1.75) -- (5.5,1.75) -- (5.5,.75) -- (5,.75) -- (5,.25);
\draw (3.5,0) -- (3.5,1.75) -- (4.5,1.75);
\node at (4.5,2) {${\scriptstyle S3^{\pi}(a)}$};
\fill[color=blue!20] (7,1.25) rectangle (7.75,1.75);
\draw (7,-.25) rectangle (7.25,1.25);
\draw (7.25,.25) -- (8,.25) -- (8,.75) -- (8.75,.75) -- (8.75,1.75) -- (7.75,1.75) -- (7.75,1.25) -- (7.25,1.25);
\draw (7,1.25) -- (7,1.75) -- (8,1.75);
\node at (8,2) {${\scriptstyle S3(a')}$};
\fill[color=blue!20] (10.25,.75) rectangle (11,1.75);
\fill[color=blue!20] (11,1.25) rectangle (11.75,1.75);
\draw (11.75,.25) rectangle (12,1.75);
\draw (10.25,-.25) -- (11.25,-.25) -- (11.25,.25) -- (11.75,.25) -- (11.75,1.25) -- (11,1.25) -- (11,.75) -- (10.25,.75) -- (10.25,-.25);
\draw (10.25,.75) -- (10.25,1.75) -- (11.75,1.75);
\node at (10.9,2) {${\scriptstyle S3^\pi(a')}$};
\end{tikzpicture}\\
\begin{tikzpicture}
\fill[color=blue!20] (0,1.25) rectangle (.75,1.75);
\draw (0,0) rectangle (1,.75) (0,.75) rectangle (1.5,1);
\draw (1.5,1) -- (2,1) -- (2,1.75) -- (.75,1.75) -- (.75,1.25) -- (0,1.25) -- (0,1);
\draw (0,1) -- (0,1.75) -- (.75,1.75);
\node at (1,2) {${\scriptstyle S3(b)}$};
\fill[color=blue!20] (3.5,.75) rectangle (4,1.75);
\fill[color=blue!20] (4,1) rectangle (4.5,1.75);
\draw (4,.75) rectangle (5.5,1) (4.5,1) rectangle (5.5,1.75);
\draw (3.5,0) -- (4.75,0) -- (4.75,.5) -- (5.5,.5) -- (5.5,.75) -- (3.5,.75) -- (3.5,0);
\draw (3.5,0) -- (3.5,1.75) -- (4.5,1.75);
\node at (4.5,2) {${\scriptstyle S3^{\pi}(b)}$};
\fill[color=blue!20] (7,1) rectangle (7.5,1.75);
\draw (7.75,.25) rectangle (8,1.75) (8,.75) rectangle (8.75,1.75);
\draw (7,-.25) -- (7.75,-.25) -- (7.75,1.75) -- (7.5,1.75) -- (7.5,1) -- (7,1) -- (7,-.25);
\draw (7,1) -- (7,1.75) -- (7.5,1.75);
\node at (8,2) {${\scriptstyle S3(b')}$};
\fill[color=blue!20] (10.25,.75) rectangle (11,1.75);
\fill[color=blue!20] (11,1.25) rectangle (11.25,1.75);
\draw (10.25,-.25) rectangle (11,.75) (11,-.25) rectangle (11.25,1.25);
\draw (11.25,1.25) -- (11.25,1.75) -- (12,1.75) -- (12,.5) -- (11.5,.5) -- (11.5,-.25) -- (11.25,-.25);
\draw (10.25,.75) -- (10.25,1.75) -- (11.25,1.75);
\node at (10.9,2) {${\scriptstyle S3^{\pi}(b')}$};
\end{tikzpicture}\\
\begin{tikzpicture}
\fill[color=blue!20] (0,.75) rectangle (.75,1.75);
\draw (0,0) rectangle (1,.25);
\draw (1,.25) -- (1.5,.25) -- (1.5,1) -- (2,1) -- (2,1.75) -- (.75,1.75) -- (.75,.75) -- (0,.75) -- (0,0);
\draw (0,0) -- (0,1.75) -- (1.75,1.75);
\node at (1,2) {${\scriptstyle S3(c)}$};
\fill[color=blue!20] (3.5,.75) rectangle (4,1.75);
\fill[color=blue!20] (4,1.5) rectangle (4.5,1.75);
\draw (4.5,1.5) rectangle (5.5,1.75);
\draw (3.5,0) -- (4.75,0) -- (4.75,1) -- (5.5,1) -- (5.5,1.5) -- (4,1.5) -- (4,.75) -- (3.5,.75) -- (3.5,0);
\draw (3.5,0) -- (3.5,1.75) -- (4.5,1.75);
\node at (4.5,2) {${\scriptstyle S3^{\pi}(c)}$};
\fill[color=blue!20] (7,1) rectangle (8,1.75);
\draw (8.5,.75) rectangle (8.75,1.75);
\draw (7,-.25) -- (7.75,-.25) -- (7.75,.25) -- (8.5,.25) -- (8.5,1.75) -- (8,1.75) -- (8,1) -- (7,1) -- (7,-.25);
\draw (7,1) -- (7,1.75) -- (8,1.75);
\node at (8,2) {${\scriptstyle S3(c')}$};
\fill[color=blue!20] (10.25,.75) rectangle (10.5,1.75);
\fill[color=blue!20] (10.5,1.25) rectangle (11.25,1.75);
\draw (10.25,-.25) rectangle (10.5,.75);
\draw (10.5,-.25) -- (11,-.25) -- (11,.5) -- (12,.5) -- (12,1.75) -- (11.25,1.75) -- (11.25,1.25) -- (10.5,1.25) -- (10.5,-.25);
\draw (10.25,.75) -- (10.25,1.75) -- (11.25,1.75);
\node at (10.9,2) {${\scriptstyle S3^{\pi}(c')}$};
\end{tikzpicture}.
\end{align*}
As before, these skew Young diagrams are arranged so that those of type ${\bf S3^{\pi}}$ are
 the $\pi$--rotations of those of type ${\bf S3}$. This time the right--hand block of six is  the
conjugation of the left--hand block of six. Thanks to the rotation symmetry and the
conjugation symmetry, we have  only to consider three cases. We
choose to be ${\bf S3(a)}$, ${\bf S3(b')}$ and ${\bf S3(c')}$.

\medskip

${\bf S3(a)}$. There are two subcases which, by Lemma \ref{bloco}, may be reduced to $(i)$ $\mu=(a^{x+1})$ and $\lambda^*=((b+c)^z,c^{x+y})$ for some integers $a,b,c,z\geq 1$ and $x,y\geq 0$ such that $x+y\geq 1$;  and $(ii)$ $\mu=((a+b)^{x+1})$ and $\lambda^*=((b+c+d)^z,c^{y+x})$ where $a,b,c,z,y\geq 1$ and $d,x\geq 0$, as illustrated below:

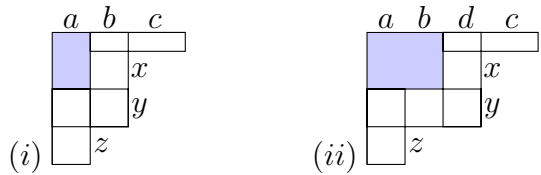
\begin{align*}
\parbox{4cm}{$(i)$\;\begin{tikzpicture}
\fill[color=blue!20] (0,1) rectangle (.5,1.75);
\draw (0,0) rectangle (.5,1) (0,.5) rectangle (1,1) (.5,.5) rectangle (1,1.75) (.5,1.5) rectangle (1.75,1.75);
\draw (0,0) -- (0,1.75) -- (1.75,1.75);
\node at (.25,1.9) {$a$};\node at (.75,1.95) {$b$};\node at (1.35,1.9) {$c$};
\node at (1.15,1.25) {$x$};\node at (1.15,.75) {$y$};\node at (.65,.25) {$z$};
\end{tikzpicture}}
\parbox{4cm}{$(ii)$\;\begin{tikzpicture}
\fill[color=blue!20] (0,1) rectangle (1,1.75);
\draw (0,0) rectangle (.5,1) (0,.5) rectangle (1.5,1) (1,.5) rectangle (1.5,1.75) (1,1.5) rectangle (2.25,1.75);
\draw (0,0) -- (0,1.75) -- (2.25,1.75);
\node at (.25,1.9) {$a$};\node at (.75,1.95) {$b$};\node at (1.3,1.95) {$d$};\node at (1.9,1.9) {$c$};
\node at (1.65,1.25) {$x$};\node at (1.65,.75) {$y$};\node at (.65,.25) {$z$};
\end{tikzpicture}}.
\end{align*}

In the subcase ${\bf S3(a)} (i)$ we start by identifying an $F0, F0^{\pi}$ or an $F0^{\pi'}$ configuration in the diagram whenever $a,z\geq 2$ or when $x\geq 1$ and $b\geq 2$, or even when $a,y\geq 2$. In these cases, the support of the skew diagram is not the entire Schur interval by Corollary \ref{TF3}. There remains thus six cases to consider.

 When $a=1$ and $x=0$ we get an $F6$, and by Lemma \ref{f6} $\sup(\lambda/\mu)=[\ww,\nn]$ if and only if $\lambda/\mu$ is an $A6$ configuration. When $a=b=1$ and  $x>0$, we get the conjugate of the $\pi$--rotation of an $\widehat{F}5$ configuration if $c\ge 2$, and  an $F2'$ configuration if $c=1$. By Proposition \ref{prop3} and Lemma \ref{f5}, we find that in these cases, $\sup(\lambda/\mu)=[\ww,\nn]$ if and only if $\lambda/\mu$ is an $A2$ configuration.

Assume now that $a\geq 2$ and $z=1$. Then, if $y=x=0$ we get a two row skew diagram, the configuration $(ii)$ of the statement under proof, and thus its support equals the Schur interval. If $y=0$ and $b=1$ we get the conjugate of an $F4$ configuration, and, by Lemma \ref{f4}, its support  is equal to  the Schur interval if and only if $\lambda/\mu$ is an $A4$ configuration. Finally, if $y=1$ and $x=0$ we get an $F2$ configuration, and when $y=1$ and $b=1$ we get the transpose of an $\widetilde{F}5$ configuration. By Proposition \ref{prop3} and Lemma \ref{f5}, we find that the support  is equal to  the Schur interval if and only if $\lambda/\mu$ is an $A2$ configuration.

Consider now the subcase ${\bf S3(a)}(ii)$. When $a,z\geq 2$, we have an $F0$ configuration, and when $x\geq 1$ and $d\geq 2$, we have an $F0^{\pi}$ configuration. Note also that if $y\geq 2$, we have an $F0'^\pi$ configuration. So we are left with six cases to analyse, all having $y=1$.

If $d=0$ and $z=1$, we get an $F2$ configuration (recall that one  restricts to  basic skew Schur functions); if $d=0$  and $a=1$, we get either an $F2$ or an $\widehat{F}5$ configuration; if $d=1$  and $z=1$, we also get the conjugate of an $\widetilde{F}5$ configuration; if $d=1$  and $a=1$, we also get the $\pi$--rotation of an $F3$ configuration; if $x=0$  and $z=1$, we also get an $F2$ configuration; and finally, if $x=0$  and $a=1$, we also get the $\pi$--rotation of an $\widehat{F}5$ configuration.
Using propositions  \ref{prop3}, \ref{f3} and Lemma \ref{f5}, it follows that the support of $\lambda/\mu$ is the full Schur interval if and only if $\lambda/\mu$ is an $A2$ or an $A3$  configuration.

\medskip

${\bf S3(b')}$. Using Lemma \ref{bloco}, we may assume that $\lambda=((a+b+1)^x,(a+1)^{y+z},a^t)$ and $\mu=(a^{x+y})$, for some integers $a,b,x,t\geq 1$, and $y,z\geq 0$ such that $y+z\geq 1$, as illustrated by
$$\lambda/\mu=\begin{tikzpicture}
\fill[color=blue!20] (0,1) rectangle (.5,2);
\draw (0,0) rectangle (.5,1) (0,.5) rectangle (.75,1) (.5,.5) rectangle (.75,2) (.5,1.5) rectangle (1.25,2);
\draw (0,0) -- (0,2) -- (1.25,2);
\node at (.25,2.15) {$a$};\node at (1,2.2) {$b$};
\node at (1.4,1.75) {$x$};\node at (.9,1.25) {$y$};\node at (.9,.75) {$z$};\node at (.65,.25) {$t$};
\end{tikzpicture}.$$
When $a,t\geq 2$ or $a,z\geq 2$ we get an $F0$ or an $F0^{\pi'}$ configuration, in which cases we know from Corollary \ref{TF3} that its support is not the entire Schur interval. An $F0'$ configuration also appears whenever $b,x\geq 2$, and again the support of $\lambda/\mu$ is strictly contained in its Schur interval. So we are left with six cases to analyse.

If $a=b=1$, we get an $F2'$ configuration, and if $a=x=1$, we get either an $F2'$ configuration (when $b=1$), or the conjugate of the $\pi$--rotation of an $\widehat{F}5$ configuration (otherwise). By Proposition \ref{prop3} and Lemma \ref{f5}, it follows that in this cases, the support of $\lambda/\mu$ is equal to the entire Schur interval if and only if $\lambda/\mu$ is an $A2$ configuration.

Assume now $a\geq 2$ and $t=z=1$. If also $b=1$, we get respectively an $F5$,  and if $x=1$, we get the conjugate of  an $\widetilde{F}5$ configuration. Again by Lemma \ref{f5}, it follows that in these cases the support of $\lambda/\mu$ is strictly contained in the Schur interval.
For the remaining two cases, assume that $a\geq 2$, $t=1$ and $z=0$. When $b=1$, we get the conjugate of an $\widetilde{F}4$ configuration, and when $x=1$, we get the conjugate of an $F4$ configuration.  In these cases, by Lemma \ref{f4}, the support of $\lambda/\mu$ is the full Schur interval if and only if $\lambda/\mu$ is the conjugate of an $A4$ configuration.

\medskip

${\bf S3(c')}$. Thanks to Lemma \ref{bloco} there are only two subcases to study: $(i)$ $\mu=(a^{x+y})$ and $\lambda^*=((b+1)^t,1^{z+y})$ for some integers $a,b,x,t\geq 1$ and $y,z\geq 0$ such that $y+z\geq 1$; and $(ii)$  $\mu=((a+b)^{x+y})$ and $\lambda^*=((b+c+1)^t,1^{z+y})$ with $a,b,x,t,z\geq 1$ and $y,c\geq 0$, as illustrated below:

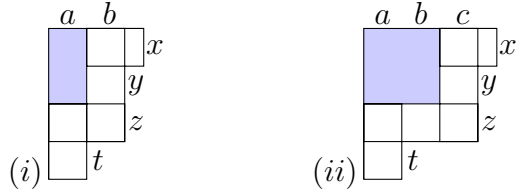
\begin{align*}
\parbox{4cm}{$(i)$\,\begin{tikzpicture}
\fill[color=blue!20] (0,1) rectangle (.5,2);
\draw (0,0) rectangle (.5,1) (0,.5) rectangle (1,1) (.5,.5) rectangle (1,2) (.5,1.5) rectangle (1.25,2);
\draw (0,0) -- (0,2) -- (1.25,2);
\node at (.25,2.15) {$a$};\node at (.8,2.2) {$b$};
\node at (1.4,1.75) {$x$};\node at (1.15,1.25) {$y$};\node at (1.15,.75) {$z$};\node at (.65,.25) {$t$};
\end{tikzpicture}}
\parbox{4cm}{$(ii)$\,\begin{tikzpicture}
\fill[color=blue!20] (-.5,1) rectangle (.5,2);
\draw (-.5,0) rectangle (0,1) (-.5,.5) rectangle (1,1) (.5,.5) rectangle (1,2) (.5,1.5) rectangle (1.25,2);
\draw (-.5,0) -- (-.5,2) -- (1.25,2);
\node at (-.25,2.15) {$a$};\node at (.25,2.2) {$b$};\node at (.8,2.15) {$c$};
\node at (1.4,1.75) {$x$};\node at (1.15,1.25) {$y$};\node at (1.15,.75) {$z$};\node at (.15,.25) {$t$};
\end{tikzpicture}}.
\end{align*}
In subcase ${\bf S3(c')} (i)$, we have $F0$ and $F0^{\pi}$ configurations whenever $a,t\geq 2$ or $b,x+y\geq 2$, respectively. In these cases, by Corollary
\ref{TF3}  we have $\sup(\lambda/\mu)\varsubsetneq[\ww,\nn]$. It remains to consider four cases. When $a=b=1$, we get an $F2'$ configuration, and when $a=x+y=1$, we get an $F6$ configuration. Then, Proposition \ref{prop3} and Lemma \ref{f6} show that, in each case, the support of $\lambda/\mu$ is equal to the Schur interval if and only if $\lambda/\mu$ is an $A2'$ or an $A6$ configuration.

Assume now $a\geq 2$ and $t=1$. If $b=1$ we get the diagram
$$
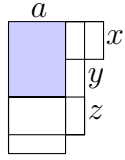
\begin{tikzpicture}
\fill[color=blue!20] (-.25,.75) rectangle (.5,1.75);
\draw (-.25,0) rectangle (0.5,.75) (-.25,.25) rectangle (.75,.75) (.5,.25) rectangle (.75,1.75) (.5,1.25) rectangle (1,1.75);
\draw (-.25,0) -- (-.25,1.75) -- (1,1.75);
\node at (.15,1.9) {$a$};
\node at (1.15,1.55) {$x$};\node at (.9,1) {$y$};\node at (.9,.55) {$z$};
\end{tikzpicture}.$$
Now, an $F0'$ appears if $z\geq 2$, and, when $z=1$, we get the conjugate of an $\widetilde{F}5$ configuration. By Theorem \ref{TF3} and Lemma \ref{f5}, it follows that in these cases the Schur interval contains strictly the  support of $\lambda/\mu$. On the other hand, if $x+y=1$, then we must have $x=1$ and $y=0$, and thus $\lambda/\mu$ has the form
$$
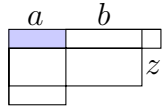
\begin{tikzpicture}
\fill[color=blue!20] (0,.75) rectangle (.75,1);
\draw (0,0) rectangle (0.75,.75) (0,.25) rectangle (1.75,.75) (.75,.75) rectangle (2,1);
\draw (1.75,.75) -- (1.75,1);
\draw (0,0) -- (0,1) -- (2,1);
\node at (.35,1.15) {$a$};\node at (1.25,1.2) {$b$};
\node at (1.9,.5) {$z$};
\end{tikzpicture}.$$
As before, an  $F0'$ appears if $z\geq 2$, and, when $z=1$, we get an $F2$ configuration, in which case, respectively, by Corollary \ref{TF3} and Proposition \ref{prop3}, we find that the Schur interval equals the support of $\lambda/\mu$ if and only it is an $A2$ configuration.

Consider now the subcase ${\bf S3(c')} (ii)$. If $a,t\geq 2$ or $z\geq 2$ or $c,x+y\geq 2$ we get respectively $F0, F0^{\pi'}$ or $F0^{\pi}$ configurations on $\lambda/\mu$. In all these cases, using Corollary \ref{TF3}, we find that $\sup(\lambda/\mu)\varsubsetneq[\ww,\nn]$.
So we may assume $z=1$.

 If  we have $c=0$, we get an $F4$ or an $\widetilde{F}4$ configuration, respectively, when $a=1$ and $t=1$, and, by Lemma \ref{f4}, the support of $\lambda/\mu$ is equal to its entire Schur interval if and only if the $\lambda/\mu$ is an $A4$ configuration.

If $c=1$ then we get either an $F5$, or an $\widetilde{F}5$ or an $F3$ configuration, respectively, when $a,x\geq 2$ and $t=1$, or when $a=1$ and $t,x\geq 2$, or when $a=t=x=1$. In these cases, by Proposition \ref{f3} and Lemma \ref{f5}, the support of the skew diagram is equal to the Schur interval if and only if $\lambda/\mu$ is an $A3$ configuration.

Finally, assume $c\geq 2$ and $x+y=1$. When $t=1$, we get an $F3$ configuration, and, when $a=1$ and $t\geq 2$, we get the $\pi$--rotation of an $\widehat{F}5$. By Proposition \ref{prop3} and Lemma \ref{f5}, we find that, in these cases, $\sup(\lambda/\mu)=[\ww,\nn]$ if and only if $\lambda/\mu$ is an $A3$ configuration.

\bigskip

 ${\bf R1.}$ There are two main subcases. We denote them by $S1$ and $S1^{\pi}$ in which $\mu$ and
$\lambda^*$, respectively, are rectangles of $m^n$--shortness 1. Each has four subcases, as illustrated
by:

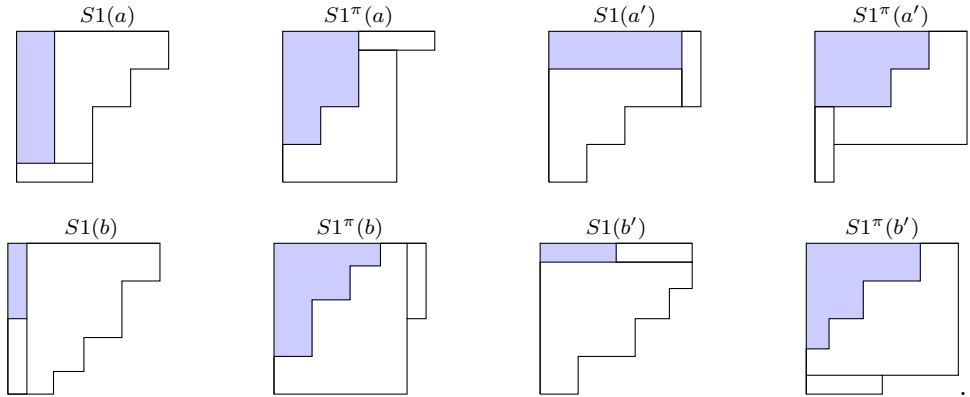
\begin{align*}
\begin{tikzpicture}
\fill[color=blue!20] (0,.25) rectangle (.5,2);
\draw (0,0) rectangle (1,.25);
\draw (.5,.25) -- (.5,2) -- (2,2) -- (2,1.5) -- (1.5,1.5) -- (1.5,1) -- (1,1) -- (1,.25);
\draw (0,0) -- (0,2) -- (2,2);
\node at (1.2,2.2) {${\scriptstyle S1(a)}$};
\fill[color=blue!20] (3.5,.5) rectangle (4,2);
\fill[color=blue!20] (4,1) rectangle (4.5,2);
\draw (4.5,1.75) rectangle (5.5,2);
\draw (3.5,0) -- (5,0) -- (5,1.75) -- (4.5,1.75) -- (4.5,1) -- (4,1) -- (4,.5) -- (3.5,.5) -- (3.5,0);
\draw (3.5,0) -- (3.5,2) -- (5.5,2);
\node at (4.5,2.2) {${\scriptstyle S1^{\pi}(a)}$};
\fill[color=blue!20] (7,1.5) rectangle (8.75,2);
\draw (8.75,1) rectangle (9,2);
\draw (7,0) -- (7.5,0) -- (7.5,.5) -- (8,.5) -- (8,1) -- (8.75,1) -- (8.75,1.5) -- (7,1.5) -- (7,0);
\draw (7,0) -- (7,2) -- (9,2);
\node at (8,2.2) {${\scriptstyle S1(a')}$};
\fill[color=blue!20] (10.5,1) rectangle (11.5,2);
\fill[color=blue!20] (11.5,1.5) rectangle (12,2);
\draw (10.5,0) rectangle (10.75,1);
\draw (10.75,.5) -- (12.5,.5) -- (12.5,2) -- (12,2) -- (12,1.5) -- (11.5,1.5) -- (11.5,1) -- (10.5,1);
\draw (10.5,0) -- (10.5,2) -- (12.5,2);
\node at (11.5,2.2) {${\scriptstyle S1^{\pi}(a')}$};
\end{tikzpicture}\\
\begin{tikzpicture}
\fill[color=blue!20] (0,1) rectangle (.25,2);
\draw (0,0) rectangle (.25,1);
\draw (.25,1) -- (.25,2) -- (2,2) -- (2,1.5) -- (1.5,1.5) -- (1.5,.75) -- (1,.75) -- (1,.3) -- (.6,.3) -- (.6,0) -- (0,0);
\draw (0,0) -- (0,2) -- (2,2);
\node at (1.1,2.2) {${\scriptstyle S1(b)}$};
\fill[color=blue!20] (3.5,.5) rectangle (4,2);
\fill[color=blue!20] (4,1.25) rectangle (4.5,2);
\fill[color=blue!20] (4.5,1.7) rectangle (4.9,2);
\draw (5.25,1) rectangle (5.5,2);
\draw (5.25,2) -- (4.9,2) -- (4.9,1.7) -- (4.5,1.7) -- (4.5,1.25) -- (4,1.25) -- (4,.5) -- (3.5,.5) -- (3.5,0) -- (5.25,0) -- (5.25,1);
\draw (3.5,0) -- (3.5,2) -- (5.5,2);
\node at (4.5,2.2) {${\scriptstyle S1^{\pi}(b)}$};
\fill[color=blue!20] (7,1.75) rectangle (8,2);
\draw (8,1.75) rectangle (9,2);
\draw (7,0) -- (7,1.75) -- (9,1.75) -- (9,1.4) -- (8.7,1.4) -- (8.7,1) -- (8.25,1) -- (8.25,.5) -- (7.5,.5) -- (7.5,0) -- (7,0);
\draw (7,0) -- (7,2) -- (9,2);
\node at (8,2.2) {${\scriptstyle S1(b')}$};
\fill[color=blue!20] (10.5,.6) rectangle (10.8,2);
\fill[color=blue!20] (10.8,1) rectangle (11.25,2);
\fill[color=blue!20] (11.25,1.5) rectangle (12,2);
\draw (10.5,0) rectangle (11.5,.25);
\draw (11.5,.25) -- (12.5,.25) -- (12.5,2) -- (12,2) -- (12,1.5) -- (11.25,1.5) -- (11.25,1) -- (10.8,1) -- (10.8,.6) -- (10.5,.6) -- (10.5,.25);
\draw (10.5,0) -- (10.5,2) -- (12.5,2);
\node at (11.5,2.2) {${\scriptstyle S1^{\pi}(b')}$};
\end{tikzpicture}.
\end{align*}
These skew Young diagrams are arranged so that those of type $S1^{\pi}$ are
the $\pi$--rotations of their left--hand neighbour of type $S1$. Moreover, the right--hand block of
four are  the conjugates of the left--hand block.
Thanks to the rotation symmetry and the
conjugation symmetry, it is therefore only necessary to consider two cases. We
select to be ${\bf S1^{\pi}(a')}$ and ${\bf S1^\pi(b)}$.

\medskip

${\bf S1^{\pi}(a')}$. Using Lemma \ref{bloco}, we may assume that $$\lambda/\mu=\begin{tikzpicture}
\fill[color=blue!20] (0,.7) rectangle (.25,2.2);
\fill[color=blue!20] (.25,1.2) rectangle (.7,2.2);
\fill[color=blue!20] (.7,1.7) rectangle (1.2,2.2);
\draw (0,0) rectangle (.25,.7);
\draw (.25,.7) -- (.25,1.2) -- (.7,1.2) -- (.7,1.7) -- (1.2,1.7) -- (1.2,2.2) -- (1.7,2.2) -- (1.7,.7) -- (.25,.7);
\draw (0,0) -- (0,2.2) -- (1.7,2.2);
\end{tikzpicture}
.$$
Moreover, we may assume that the diagram has at least 4 columns, three amongst the last but one have different sizes,  otherwise we are in case ${\bf R3}$ or ${\bf R4}$. In this case, $\lambda/\mu$ is disconnected and the largest component as a 2 by 2 block. By Lemma \ref{disconnected}, it follows that the support of $\lambda/\mu$ is not the full Schur interval.

\medskip

${\bf S1^{\pi}(b)}$. We may use again Lemma \ref{bloco} to reduce our study to the skew diagrams of the form $$\lambda/\mu=\begin{tikzpicture}
\fill[color=blue!20] (1.2,2.2) rectangle (1.7,2.8);
\fill[color=blue!20] (.25,1.2) rectangle (.7,2.8);
\fill[color=blue!20] (.7,1.7) rectangle (1.2,2.8);
\draw (1.7,1.2) rectangle (1.95,2.8);
\draw (.25,.7) -- (.25,1.2) -- (.7,1.2) -- (.7,1.7) -- (1.2,1.7) -- (1.2,2.2) -- (1.7,2.2) -- (1.7,.7) -- (.25,.7);
\draw (.25,.7) -- (.25,2.8) -- (1.7,2.8);
\node at (1,2.95) {$a$};\node at (2.1,2) {$x$}; \node at (1.85,.9) {$y$};
\end{tikzpicture},$$
\noindent where $a+1$ is the number of columns of the diagram, which we assume that amongst the first $a$ there is at least three distinct lengths (otherwise we are in previous cases), $x$ is the length of the last column and  $\lambda^*=(y)$. Moreover, we assume without loss of generality that the last but one column starts at least one row below the topmost box of the last column, and that the first column has length $\leq y$.
Note that if there are at least two columns, among the first $a$ ones, having the same length and $y\geq 2$, then $\lambda/\mu$ is an $F0$ configuration.  We are left with two cases: either $(i)$ $y=1$, or $(ii)$ the first $a$ columns have pairwise distinct lengths.
In the first case, the minimal and maximal fillings of $\lambda/\mu$ are
$$\ww=(x,n_1^{s_1},\ldots,n_r^{s_r},1^{s_{r+1}})\preceq\nn=(x+1,n_1^{s_1},\ldots,n_r^{s_r},1^{s_{r+1}-1}),$$
with $x\geq n_1>\cdots>n_r>1$ and $r\geq 2$.
Therefore, the partition $$\xi:=(x,n_1^{s_1},\ldots,n_r+1,n_r^{s_r-1},1^{s_{r+1}-1})$$ shows that the support of $\lambda/\mu$ is not the entire admissible interval.

Finally, assume that amongst the first $a$ columns of $\lambda/\mu$ there are no two with the same length, and that $y\geq 2$.
By Corollary \ref{f8}  it follows that the support of $\lambda/\mu$ is equal to the Schur interval if and only if $\lambda/\mu$ is an $A7$ configuration.
$\square$

\medskip

\begin{example}
Let $\lambda/\mu=((b+3)^2,2^{y+1})/(b+2,1^{y+1})$ be an $A3$ configuration
with $\ww=(y+2,2,1^{b+1})\preceq\nn=(y+3,2,1^b)$.
Then
$$s_{\lambda/\mu}=s_{(y+2,2,1^{b+1})'}+s_{(y+2,2,2,1^{b-1})'}+s_{(y+2,3,1^{b})'}+s_{(y+3,1^{b+2})'}+s_{(y+3,2,1^b)'}.$$
\end{example}

\begin{example}
Let $\lambda/\mu=(4^4,3^2)/(3,2,1^3)$ be an $A7$ configuration with
$\ww=(5,4,4,1)\preceq\nn=(6,5,3)$.
Then
$$s_{\lambda/\mu}=s_{(5,4,4,1)'}+s_{(5,5,3,1)'}+s_{(6,4,3,1)'}+s_{(5,5,4)'}+s_{(6,5,2,1)'}+s_{(6,4,4)'}+s_{(6,5,3)'}.$$
\end{example}

\bigskip

The characterisation of the multiplicity--free  Schur function products that attain the full interval, given in Corollary \ref{MainC}, is now a consequence of  Theorem  \ref{main},  and of Corollary \ref{stemp}.

\medskip
We list now explicitly the partitions $(\mu,\nu,\lambda)$ for which $c_{\mu\,\nu}^\lambda=1$ for all  $\lambda\in[\mu\cup\nu,\,\mu+\nu]$.
Recall that the Pieri rule expresses the product of a Schur function and a single row (column) Schur function in terms of Schur
functions \cite{Pieri, ful,stanley}. These are precisely the cases where the Hasse diagram of the interval $[\mu\cup\nu,\,\mu+\nu]$ is given by the Pieri rule.
\begin{corollary}
Let $(\mu,\,\nu,\lambda)$ be a triple of partitions.
\begin{itemize}
\item[$(a)$] If $\mu$ or $\nu$ is the zero partition,
$c_{\mu,0}^\lambda=1$ if and only if $\mu=\lambda$

\item[$(b)$] If $\mu=(1^x)$ and $\nu=(1^y)$ (or { vice versa}), with $x\ge y\ge 1$,

$c_{\mu,\nu}^\lambda=1$ if and only if $\lambda\in [(1^{x+y});\,(2^y,\,1^{x-y})]$.

\item[$(b')$] If $\mu=(x)$ and $\nu=(y)$ (or { vice versa}), with $x\ge y\ge 1$,

$c_{\mu,\nu}^\lambda=1$ if and only if $\lambda\in [({x,\,y});\,(x+y)]$.

\item[$(c)$] If $\mu=(1^x)$ and  $\nu=(2,1^y)$
is  such that  $1\le x\leq y+1$ (or { vice versa}),

$c_{\mu,\nu}^\lambda=1$ if and only if $\lambda\in [(2,\,1^{x+y});\,(3,\,2^{x-1},\,1^{y-x+1})]$.

\item[$(c')$]  If $\mu=(x)$  and  $\nu=(z,1)$
is such that $1\le x\leq z$,
 (or { vice versa}),

$c_{\mu,\nu}^\lambda=1$ if and only if $\lambda\in [(z,\,x,\,1);\,(z+x+1,1)]$.

\item[$(d)$] If $\mu=(1)$ and  $\nu=(a,1^y)$
is  such that $a\ge 3$, $y\ge 1$ (or { vice versa}),

$c_{\mu,\nu}^\lambda=1$ if and only if $\lambda\in [(a,\,1^{y+1});\,(a+1,\,1^y)]$.

\item[$(d')$]  If $\mu=(1)$  and  $\nu=(z,1^a)$
is such that $a\ge 2$, $z\ge 1$
 (or { vice versa}),

$c_{\mu,\nu}^\lambda=1$ if and only if $\lambda\in [(z,\,1^{a+1});\,(z+1,\,1^a)]$.
\end{itemize}

\end{corollary}

\begin{remark} We may also list explicitly  the multiplicity--free Schur function products  whose support is an interval, that is, those whose number of components (summands) is the cardinal
of the Schur interval.

\begin{enumerate}

\item[$(a)$]
$ s_0s_\nu=s_\nu$ has 1 component;

\vskip0.3cm

\item[$(b)$] The conjugate  Schur interval $[(1^{x+y});\,(2^y,\,1^{x-y})]$ is a saturated chain.
$$s_{(1^x)}s_{(1^y)}=s_{(2^y,\,1^{x-y})}+s_{(2^{y-1},\,1^{x-y+2})}+\cdots+s_{(2^{1},\,1^{x+y-2})}+s_{(1^{y+x})},\quad x\ge y\ge 1,$$
 \noindent has $y+1$ components. In particular,
 when $y=1$,  $s_{(1^x)}s_{(1)}=s_{(2,\,1^{x-1})}+s_{(1^{1+x})}$ has 2 components.

\vskip0.3cm

\item[$(b')$] It is the conjugate of $(b)$.
\vskip0.3cm

\item[$(c)$] There are two cases for the conjugate Schur interval $[(2,\,1^{x+y}),\,(3,\,2^{x-1},\,1^{y-x+1})]$, with $1\le x\le y+1$.
When $y=x-1$,

$$s_{1^x}s_{2\,1^{x-1}}=s_{(x\,x\,1)'}+s_{(x+1\,x-1\,1)'}+s_{(x+1\,x)'}+s_{(x+2\,x-2\,1)'}+s_{(x+2\,x-1)'}+s_{(x+3\,x-3\,1)'}+\cdots$$
$$+s_{(2x-1\,1\,1)'}+s_{(2x-1\,2)'}+s_{(2x\,1)'},\;x\ge 1;\;\mbox{and}$$

\noindent when $y> x-1$,
$$s_{(2,\,1^{x-1+k})}s_{(1^x)}=s_{(x+k,\;x,\;1)'}+s_{(x+k,\;x+1)'}+s_{(x+1+k,\;x-1,\;1)'}+s_{(x+1+k,\;x)'}+
s_{(x+2+k,\;x-2,\;1)'}+$$ $$+s_{(x+2+k,\;x-1)'}
+s_{(x+3+k,\;x-3,\;1)'}+\cdots+s_{(2x-1+k,\;1,\;1)'}+s_{(2x-1+k,\;2)'}+s_{(2x+k,\;1)'},\;\;x\ge 1,\;k> 0;$$

\vskip0.3cm

\item[$(c' )$] It is the conjugate of $(c)$.

\vskip0.3cm

\item[$(d )$] The conjugate Schur interval $[(a,\,1^{y+1});\,(a+1,\,1^y)]$, $a\ge 3$, $y\ge 1$, has 3 elements.

 $$s_{(a,\,1^y)}s_1=s_{({a},\,1^{y+1})}+s_{(a,\,2,\,1^{y-1})}+s_{(a+1,\,1^{y})},\quad a\ge 3,\;y\ge 1.$$

\item[$(d' )$] It is the conjugate of $(d)$.

\end{enumerate}
\end{remark}

The classification of the   Schur function products that attain the full interval follows from Proposition \ref{rib} and from the study of the multiplicity-free case.
\begin{corollary}\label{cor:fullschur}
The Schur function product $s_{\mu}s_{\nu}$ has all Littlewood-Richardson  coefficients positive in the entire Schur interval if and only if one of the conditions in  Corollary \ref{MainC},  or one of the following is true:
$\mu=(r_1,1^{r_2})$ and $\nu=(s_1,1^{s_2})$ are hooks such that $s_2=r_2=1$, and either $r_1=s_1\ge 2$ or $r_1\geq2$ and $s_1=r_1+1$ (or vice versa).
\end{corollary}

$$\parbox{3cm}{\begin{tikzpicture}
\draw (0,0) -- (0,1)--(1.25,1)--(1.25,0.8)--(1,.8)--(1,.4)--(.4,.4)--(.4,0)--(0,0);
\end{tikzpicture}}\parbox{5cm}{\begin{tikzpicture}
\fill[color=blue!20] (0,.25) rectangle (1,.5);\draw (0,.25) rectangle (1,.5);
\draw (0,0) rectangle (1,.25);\draw (1,.25) rectangle (2.25,.5);
\end{tikzpicture}}\parbox{6cm}{\begin{tikzpicture}
\fill[color=blue!20] (0,.75) rectangle (.85,1.6);\draw (0,.75) rectangle (.85,1.6);
\draw (0,0) rectangle (0.25,.75);\draw (0,.5) rectangle (.85,.75);\draw (.85,.75) rectangle (1.1,1.6);
\node at (.5,1.75) {$a$};\node at (1.25,1.2) {$x$};\node at (.4,.25) {$y$};\node at (4,1.2) {$a=2 \text{ and }1\le x\leq y+1,$};
\node at (3.4,.8) {$\text{ or }a\geq 3\text{  and } x=1;$};
\end{tikzpicture}}$$

$$\parbox{7cm}{\begin{tikzpicture}
\fill[color=blue!20] (0,.75) rectangle (.85,1);\draw (0,.75) rectangle (.85,1);
\draw (0,0) rectangle (0.25,.75);\draw (0,.5) rectangle (.85,.75);\draw (.85,.75) rectangle (1.75,1);
\node at (.5,1.15) {$z$};\node at (1.25,1.15) {$x$};\node at (.4,.25) {$a$};\node at (4.1,1) {$a=1\text{ and }1\le x\leq z,$};
\node at (3.9,.6) {$\text{  or }a\geq 2\text{ and } x=1;$};
\end{tikzpicture}}\parbox{7cm}{\begin{tikzpicture}
\fill[color=blue!20] (0,.75) rectangle (.85,1.6);\draw (0,.75) rectangle (.85,1.6);\draw (.85,1.35) rectangle (1.7,1.6);
\draw (0,0) rectangle (0.25,.75);\draw (0,.5) rectangle (.85,.75);\draw (.85,.75) rectangle (1.1,1.6);
\node at (.5,1.75) {$r_1$};\node at (1.25,1.75) {$s_1$};\node at (.5,.25) {$r_2$};\node at (1.35,1) {$s_2$};
\node at (4.5,1.2) {$s_2=r_2=1,\text{ and either }$};
\node at (4.25,.8) {$r_1=s_1\ge 2\text{ or }r_1\geq2$};
\node at (3.8,.4) {$\text{and }s_1=r_1+1.$};
\end{tikzpicture}}
$$

\begin{proof}
By Proposition \ref{rib}, Corollary \ref{MainC}, and Proposition \ref{f4},  it remains to analyse the case
where both partitions $\mu=(r_1,1^{r_2})$ and $\nu=(s_1,1^{s_2})$  are hooks. The Schur interval of the product $s_{\mu}s_{\nu}$
is $[\ww=(a,b,1^{r_1+s_1-2}),\nn=(r_2+s_2+2,2^{d-1},1^{c-d})]$, where $a=\max\{r_2+1,s_2+1\}$, $b=\min\{r_2+1,s_2+1\}$, and $c=\max\{r_1,s_1\}$, $d=\min\{r_1,s_1\}$.

 We start by considering that the two hooks are in the conditions of the corollary, that is,  the arms  $r_1\le s_1$ of both hooks differ by at most one unit and both legs $r_2+1$ and $s_2+1$ have length $2$ ($r_2=s_2=1$).  Then $[\ww=(2,2,1^{r_1+s_1-2}),\nn=(4,2^{r_1-1},1^{s_1-r_1})]$. Denote by  $A$  the skew shape of the corresponding diagram consisting of the two hooks. The size  of $A$ is $|A|=r_1+s_1+2$.  Given a partition $\ww\precneqq\rho=(\rho_1,\ldots,\rho_{\ell})\precneqq\nn$, we must have
 $\rho_1\in\{2,3,4\}$.

  If $\rho_1=4$, then by definition of dominance order $0<k<r_1-1$ boxes in the rows of length two of $\nn$ have to lower down to get  $\rho=(4,2^{r_1-1-k},1^{s_1-r_1+2k})$.
 Let $T$ be the only tableau with shape  $A$ and content $\nn$. Then, changing the leftmost $k$ letters $2$ in the third row of $T$ into letters $1$, we get a tableau with skew shape $A$ and content $\rho$, proving that $\rho$ is in the support of $s_{\mu}s_{\nu}$.

 If $\rho_1=2$, then, by the reverse argument on the dominance order, $0<k<r_1-1$ boxes in the rows of lengths one of $\ww$ have to be lifted to get $\rho=(2^{2+k},1^{r_1+s_1-2-2k})$.  Let $T$ be the only tableau with shape  $A$  and content $\ww$. Then, changing the rightmost $k$ letters $1$ in the third row of $T$ into letters $2$, we get a tableau with shape $A$ and content $\rho$, proving that $\rho$ is in the support of $s_{\mu}s_{\nu}$.

Assume now that $\rho_1=3$. Then, $\rho_2\in\{2,3\}$ but $\rho_3<3$ since the sum of the first three parts of $\nn$ is $n_1+n_2+n_3\le 4+2+2<9$. If $\rho_2=2$, then
 $\rho=(3,2^{k},1^{k'})$ with $0\leq k\leq r_1$ and  $k'=r_1+s_1+2-3-2k$.  Fill the boxes  of the shape $A$ by placing letters $1$ in the first row, a letter $2$ in the second row and a letter $3$ in the forth; place also $k$ letters $2$ in the third row and fill the remaining boxes of this row with $1$'s. The resulting tableau is an LR tableau with content $\rho$.

  Finally, if $\rho_2=3$, then $\rho=(3,3,2^{k},1^{k'})$ with $0\leq k\leq r_1-2$ and $k'=r_1+s_1+2-6-2k$.
Fill the boxes of the skew shape $A$ by placing letters $1$ in the first row, a letter $2$ in the second row and a letter $3$ in the forth and in the rightmost box of the third row; place also $k$ letters $2$ in the third row and fill the remaining boxes of this row with $1$'s. The resulting tableau is an LR tableau with content $\rho$.

Assume now that the two hooks are not in the conditions of the corollary. Then, either one of the legs has length greater than 2, or the two arms differ by at least two boxes. That is, either $(i)$ $r_2>1$ or $s_2>1$ and $r_1,s_1\geq 3$, or $(ii)$ $r_2=s_2=1$, $r_1\geq 2$ and $s_1>r_1+1$.
In case $(i)$, we may assume without loss of generality that $r_2\geq s_2$, and if $r_2=s_2$ consider the partition  $\rho=(r_2+1, r_2+1, 3, 1^{r_1+s_1-5})$, and if $r_2>s_2>1$ consider $\rho=(r_2+1,s_2+2,3, 1^{r_1+s_1-6})$, and  we have no place in $A$ for a string of length $3$. In case $(ii)$, consider the partition $\rho=(2,2,2^{r_1-1},2,1^{s_1-r_1-2})$  and we have at most place in $A$ for $r_1+1$ strings of length $2$. In any case, it is clear that $\rho$ is in the interval $[\ww,\nn]$, but not in the support of $s_{\mu}s_{\nu}$.
\end{proof}

\vskip 0.3cm

\noindent{\em Acknowledgments}: The first and third authors would like to express their gratitude to P. R. W. McNamara for the interesting discussions during the FPSAC 2009, Hagenberg, Austria, and in particular for recalling to the first author the algorithm in~\cite{az}.
The second author would like to thank CMUC, Centre for Mathematics, University of Coimbra, for offering him hospitality during the preparation of this paper.

\end{document}